\documentclass[12pt,]{amsart}
\usepackage[latin1]{inputenc}
\usepackage{indentfirst}
\usepackage[dvips]{graphicx}
\usepackage{amsfonts}
\usepackage{amssymb}
\usepackage{theoremref}
\usepackage{bbold}
\usepackage{dsfont}
\usepackage{amsmath}
\usepackage{pb-diagram}
\usepackage{geometry}
\usepackage{pstricks,pst-node}
\usepackage{lmodern}
\usepackage{etoolbox}
\apptocmd{\sloppy}{\hbadness 10000\relax}{}{}
\newcommand {\nl}{\newline}

\begin{document}

\newtheorem{theorem}{Theorem}[section]
\newtheorem{lemma}[theorem]{Lemma}
\newtheorem{corollary}[theorem]{Corollary}
\newtheorem{proposition}[theorem]{Proposition}
\newtheorem{remark}[theorem]{Remark}
\newtheorem{definition}[theorem]{Definition}
\newtheorem{example}[theorem]{Example}

\newcommand {\afi}{\textbf{Claim: }}
\newcommand {\prafi}{\textbf{Proof of claim: } }
\newcommand {\D}{\displaystyle}
\newcommand {\eps}{\varepsilon}

\title[Universal C$^*$-algebras and partial actions]{C$^*$-algebras of endomorphisms of groups with finite cokernel and partial actions}

\author{Felipe Vieira}\thanks{Supported by CAPES - Coordena\c{c}\~{a}o de Aperfei\c{c}oamento de Pessoal de N\'{i}vel Superior}

\begin{abstract}
In this paper we extend the constructions of Boava and Exel to present the C$^*$-algebra associated with an injective
endomorphism of a group with finite cokernel as a partial group algebra and consequently as a partial crossed product. With this representation we present another way to study such C$^*$-algebras, only using tools from partial crossed products.
\end{abstract}

\maketitle

\vspace{1cm}

\section{Introduction}

Consider an injective endomorphism $\varphi$ of a discrete countable group $G$ with unit $\{e\}$ with finite cokernel i.e,
\begin{equation}\label{eqintro1}
\left|\dfrac{G}{\varphi(G)} \right|<\infty,
\end{equation}

as above, for $H$ subgroup of $G$, we use $\frac{G}{H}$ to denote the set of left cosets of $H$ in $G$. Analyzing the natural representation of $G$ and $\varphi$ inside $\mathcal{L}(l^2(G))$ we construct a concrete C$^*$-algebra $C_r^*[\varphi]\subseteq \mathcal{L}(l^2(G))$ and a universal one denoted by $\mathds{U}[\varphi]$. Such constructions were presented by Hirshberg in \cite{Hirsh}, and were later generalized by Cuntz and Vershik in \cite{CunVer} and also in \cite{Viei}.

Using a semigroup crossed product description of $\mathds{U}[\varphi]$ implies the existence of a (full corner) group crossed product description of it (\cite{Cuntz2}, \cite{CuntzTopMarkovII} and \cite{Laca1}), but it is not the only way to represent it as a crossed product: analogously to the work of G. Boava and R. Exel in \cite{BoEx} one can show that $\mathds{U}[\varphi]$ has a partial group crossed product description, which can also be related to an inverse semigroup crossed product by \cite{ExVi}.

We present in this paper the latter construction cited above and show the simplicity of $\mathds{U}[\varphi]$, which is part of the conclusions in \cite{Hirsh}, using only the partial group crossed product description of that C$^*$-algebra.

\section{Definition}

We repeat the constructions of \cite{Hirsh}. Let $G$ be a discrete countable group with unit $e$ and $\varphi$ an injective endomorphism (monomorphism) of $G$ with finite cokernel (\ref{eqintro1}).

Consider the Hilbert space $l^2(G)$ with orthonormal basis $\xi_h$, taking every element of $G$ to 0 apart from the element $h$, which goes to 1. Define the following bounded operators on $l^2(G)$:
$$
U_g(\xi_h)=\xi_{gh}
$$

and
$$
S(\xi_g)=\xi_{\varphi(g)}.
$$

The invertibility property of groups and the injectivity of the endomorphism $\varphi$ imply that the $U_g$'s are unitary operators and $S$ is an isometry respectively. Therefore we define the following C$^*$-algebra.
\begin{definition}\label{defi1red}We denote $C_r^*[\varphi]$ the reduced C$^*$-algebra of $\varphi$, to be the C$^*$-subalgebra of $\mathcal{L}(l^2(G))$ generated by the above defined unitaries $\{U_g:\;g\in G\}$ and isometry $S$.
\end{definition}

Inspired by the properties of the operators above:
\begin{definition}\label{defi1}We call $\mathds{U}[\varphi]$ the universal C$^*$-algebra generated by the unitaries $\{u_g:\;g\in G\}$ and one isometry $s$ such that:
\begin{enumerate}
  \item[(i)] $u_gu_h=u_{gh}$;
  \item[(ii)] $su_g=u_{\varphi(g)}s$;
  \item[(iii)] $\D\sum_{g\in G/\varphi(G)}u_gss^*u_{g^{-1}}=1$;
\end{enumerate}
for all $g$, $h\in G$.
\end{definition}

As the universal C$^*$-algebra above is defined using relations satisfied by the generators of the reduced one, obviously there is a canonical surjective $*$-homomorphism from $\mathds{U}[\varphi]$ onto $C_r^*[\varphi]$.

Note that the conditions (i) and (ii) above can be merged into the relation
$$
u_gs^nu_hs^m=u_{g\varphi^n(h)}s^{n+m}.
$$

By (ii) we have, for $g\in G$, the obvious relations
$$
u_{g}s^*=s^*u_{\varphi(g)}
$$

and
$$
u_{\varphi(g)}ss^*=ss^*u_{\varphi(g)}.
$$

Also note that in (iii) there is no ambiguity if we choose different representatives of the cosets:
\begin{equation*}
\begin{split}
u_{g\varphi(h)}ss^*u_{(g\varphi(h))^{-1}}&=u_gu_{\varphi(h)}ss^*u_{\varphi(h^{-1})}u_{g^{-1}}=u_gss^*u_{\varphi(h)}u_{\varphi(h^{-1})}u_{g^{-1}}\\
&=u_gss^*u_{g^{-1}}.
\end{split}
\end{equation*}

Condition (iii) implies that $u_gss^*u_{g^{-1}}$ and $u_hss^*u_{h^{-1}}$ are orthogonal projections if $g^{-1}h\notin\varphi(G)$, so the multiplication can be described as:
\begin{equation*}
u_gss^*u_{g^{-1}}u_hss^*u_{h^{-1}}=\left\{
                                     \begin{array}{cl}
                                       u_gss^*u_{g^{-1}}, & \hbox{if }h\in g\varphi(G); \\
                                       0, & \hbox{otherwise.}
                                     \end{array}
                                   \right.
\end{equation*}

This extends to the family of elements of type $u_gs^n{s^*}^nu_{g^{-1}}$ for any $n\in\mathds{N}$
\begin{equation*}
u_gs^n{s^*}^nu_{g^{-1}}u_hs^n{s^*}^nu_{h^{-1}}=\left\{
                                     \begin{array}{cl}
                                       u_gs^n{s^*}^nu_{g^{-1}}, & \hbox{if }h\in g\varphi^n(G); \\
                                       0, & \hbox{otherwise.}
                                     \end{array}
                                   \right.
\end{equation*}

And note that, for $g$, $h\in G$ and $n\geq m\in\mathds{N}$:
\begin{equation*}
\begin{split}
&u_gs^n{s^*}^nu_{g^{-1}}u_hs^m{s^*}^mu_{h^{-1}}\\
&=u_gs^n{s^*}^nu_{g^{-1}}u_hs^m\left(\D\sum_{k\in\frac{G}{\varphi^{n-m}(G)}}u_ks^{n-m}{s^*}^{n-m}u_{k^{-1}}\right){s^*}^mu_{h^{-1}}\\
&=u_gs^n{s^*}^nu_{g^{-1}}\left(\D\sum_{k\in\frac{G}{\varphi^{n-m}(G)}}u_{h\varphi^m(k)}s^n{s^*}^nu_{(h\varphi^m(k))^{-1}}\right)\\
&=\left\{
    \begin{array}{ll}
      u_gs^n{s^*}^nu_{g^{-1}},&\hbox{ if }h\varphi^m(k)\in g\varphi^n(G)\hbox{ for some }k\in\frac{G}{\varphi^{n-m}(G)}; \\
      0,\hbox{ otherwise}.
    \end{array}
  \right.
\end{split}
\end{equation*}

\section{Crossed product description of $\mathds{U}[\varphi]$}\label{sectioncpdescr}

In this section we present a semigroup crossed product description of $\mathds{U}[\varphi]$. The semigroup crossed product definition which we will use is the same as presented in Appendix A of \cite{Li1}, via covariant representations. In our case the semigroup implementing the action will be the semidirect product
$$
S:=G\rtimes_\varphi\mathds{N}=\{(g,n)\;:\;g\in G, n\in\mathds{N}\}
$$

with product
$$
(g,n)(h,m)=(g\varphi^n(h),n+m).
$$

We will also show that the action implemented by $S$ can be split i.e, the semigroup crossed product by $S$ can be seen as a semigroup crossed product by $\mathds{N}$. This crossed product description is a great tool to prove some properties of $\mathds{U}[\varphi]$: we will show that when $G$ is amenable this C$^*$-algebra is nuclear and satisfies UCT. Secondly, that description allows one to use the six-term exact sequence introduced by M. Khoshkam and G. Skandalis in \cite{Khoska} on $\mathds{U}[\varphi]$.

Moreover, due to M. Laca \cite{Laca1}, sometimes it is possible to see semigroup crossed products as full corners of group ones, which implies that both are Morita equivalent and therefore have the same K-groups. And in case the semigroup action is implemented by $\mathds{N}$, Laca's dilation turns this $\mathds{N}$-action into a $\mathds{Z}$-action, which fits the requirements to use the classical Pimsner-Voiculescu exact sequence \cite{Pivo1}.

Set
$$
\overline{G}:=\D\lim_{\leftarrow}\left\{\frac{G}{\varphi^m(G)}: p_{m,l+m}\right\}
$$

where
$$
p_{m,l+m}:\dfrac{G}{\varphi^{l+m}(G)}\rightarrow\dfrac {G}{\varphi^m(G)}
$$

is the canonical projection. We can see $\overline{G}$ as
$$
\overline{G}=\left\{(g_m)_m\in\D\prod_{m\in\mathds{N}}\frac{G}{\varphi^m(G)}:\;p_{m,l+m}(g_{l+m})=g_m\right\},
$$

with the induced topology on the product $\D\prod_{m\in\mathds{N}}\frac{G}{\varphi^m(G)}$, where each finite set $\dfrac{G}{\varphi^m(G)}$ carries the discrete topology, implying that $\overline{G}$ is a compact space.\\[2\baselineskip]
Furthermore, we have the map
\begin{equation*}
\begin{split}
G&\rightarrow\overline{G}\\
g&\mapsto (g)_m,
\end{split}
\end{equation*}

which is an embedding when $\varphi$ is pure. Also set
$$
\mathcal{G}:=\lim_{\rightarrow}\{\mathcal{G}_m:\phi_{l+m,m}\}
$$

where $\mathcal{G}_m=\overline{G}$ for all $m\in\mathds{N}$ and $\phi_{l+m,m}=\varphi^l$. We can see $\mathcal{G}$ as
$$
\mathcal{G}=\D\bigcup^._{m\in\mathds{N}}\mathcal{G}_m\diagup\thicksim
$$

with $x_l\sim y_m$ if and only if $\varphi^m(x_l)=\varphi^l(y_m)$, $x_l\in\mathcal{G}_l$ and $y_m\in\mathcal{G}_m$. Note that $\mathcal{G}$ is a locally compact set.

Denote by $q$ the canonical projection
$$
q:\D\bigcup_{m\in\mathds{N}}^.\mathcal{G}_m\rightarrow\mathcal{G},
$$

and $i_m$ the embedding
$$
\begin{array}{cccccc}
  i_m: & \overline{G} & = & \mathcal{G}_m & \hookrightarrow & \mathcal{G} \\
       &  x           & = & x             & \mapsto         & q(x).
\end{array}
$$

Again we have the identification
\begin{equation*}
\begin{split}
\overline{G}&\hookrightarrow\mathcal{G}\\
x&\mapsto i_0(x).
\end{split}
\end{equation*}
\begin{remark}\em\label{obs12} Note that if we suppose that our endomorphism $\varphi$ is \emph{totally normal}, i.e. all the $\varphi^m(G)$ are normal subgroups of $G$, then $\overline{G}$ and $\mathcal{G}$ will be groups; one just has to consider the componentwise multiplication in $\overline{G}$ and
$$
i_m(x)i_l(y)=i_{l+m}(xy),\; \forall\; x, y\in\overline{G}
$$

on $\mathcal{G}$.
\end{remark}

\begin{proposition}\label{prop3}The map
\begin{equation*}
\begin{split}
\alpha: C^*(P)&\rightarrow C(\overline{G})\\
u_gs^n{s^*}^nu_{g^{-1}}&\mapsto p_{g\varphi^n(\overline{G})},
\end{split}
\end{equation*}

where the latter denotes the characteristic function on the subset $g\varphi^n(\overline{G})\subseteq\overline{G}$, is an isomorphism.
\end{proposition}
\begin{proof} It is clear that $C^*(P)$ is the inductive limit of
$$
D_m:=C^*\left(\left\{u_gs^m{s^*}^mu_{g^{-1}}:g\in\frac{G}{\varphi^m(G)}\right\}\right)
$$

with the inclusions (using (iii) of Definition \ref{defi1})
\begin{equation*}
\begin{split}
D_m                    &\hookrightarrow D_{l+m}\\
u_gs^m{s^*}^mu_{g^{-1}}&\mapsto\D\sum_{h\in\frac{G}{\varphi^l(G)}}u_{g\varphi^m(h)}s^{l+m}{s^*}^{l+m}u_{\varphi^m(h^{-1})g^{-1}}.
\end{split}
\end{equation*}

Furthermore the pairwise orthogonality of the projections $u_gs^m{s^*}^mu_{g^{-1}}$ for fixed $m\in\mathds{N}$ implies that
$$
spec(D_m)\cong\dfrac{G}{\varphi^m(G)}
$$

with
\begin{equation*}
\begin{split}
spec(D_{l+m})&\rightarrow spec(D_m)\\
\chi         &\mapsto\chi|_{D_m}
\end{split}
\end{equation*}

corresponding to
\begin{equation*}
\begin{split}
p_{l+m,m}:\frac{G}{\varphi^{l+m}(G)}&\rightarrow\frac{G}{\varphi^m(G)}\\
g\varphi^{l+m}(G)                   &\mapsto g\varphi^m(G).
\end{split}
\end{equation*}

Therefore
$$
spec(C^*(P))\cong\D\lim_{\leftarrow}\left\{\frac{G}{\varphi^m(G)}:p_{m,l+m}\right\}=\overline{G}.
$$

Thus we get the isomorphism
\begin{equation*}
\begin{split}
\alpha:C^*(P)          &\rightarrow C(\overline{G})\\
u_gs^m{s^*}^mu_{g^{-1}}&\mapsto p_{g\varphi^m(\overline{G})}.
\end{split}
\end{equation*}
\end{proof}

\begin{definition}\label{defi2}The stabilization of \;$\mathds{U}[\varphi]$, denoted by $\mathds{U}^s[\varphi]$, is the inductive limit of the system $\{\mathds{U}_m^s[\varphi]:\psi_{m,l+m}\}$ where, $\forall\; m\in\mathds{N}$, $\mathds{U}_m^s[\varphi]=\mathds{U}[\varphi]$ and
\begin{equation*}
\begin{split}
\psi_{m,l+m}:\mathds{U}[\varphi]&\rightarrow\mathds{U}[\varphi]\\
															 x&\mapsto s^lx{s^*}^l.
\end{split}
\end{equation*}

Furthermore define $C^*(P)^s=\D\lim_{\rightarrow}\{C^*(P)_m^s:\psi_{m,l+m}\}$ with $C^*(P)_m^s=C^*(P)$ and $\psi_{m,l+m}$ as above.
\end{definition}
\begin{proposition}\label{prop4}We have $C^*(P)^s\cong C_0(\mathcal{G})$.
\end{proposition}
\begin{proof} The maps $\psi_{m,l+m}$, conjugated by $\alpha$, give maps
$$
\widetilde{\psi}_{m,l+m}:=\alpha\circ\psi_{m,l+m}\circ\alpha^{-1}: C(\overline{G})\rightarrow C(\overline{G}),
$$

where $\widetilde{\psi}_{m,l+m}(f)(x)=f(\varphi^{-l}(x))p_{\varphi^l(\overline{G})}(x):$
\begin{equation*}
\begin{split}
\widetilde{\psi}_{m,l+m}(p_{g\varphi^m(\overline{G})})(x)&=\widetilde{\psi}_{m,l+m}\circ\alpha(u_gs^m{s^*}^mu_{g^{-1}})(x)\\
&=\alpha\circ\psi_{m,l+m}(u_gs^m{s^*}^mu_{g^{-1}})(x)\\
&=\alpha(u_{\varphi^l(g)}s^{l+m}{s^*}^{l+m}u_{\varphi^l(g^{-1})})(x)\\
&=p_{\varphi^l(g)\varphi^{l+m}(\overline{G})}(x)\\
&=p_{g\varphi^m(\overline{G})}(\varphi^{-l}(x))p_{\varphi^l(\overline{G})}(x).
\end{split}
\end{equation*}

By the properties of inductive limits, we have an isomorphism
$$
\overline{\alpha}: C^*(P)^s\rightarrow\D\lim_{\rightarrow}\{C(\overline{G}):\widetilde{\psi}_{m,l+m}\}.
$$

Additionally we consider the $*$-homomorphisms
\begin{equation*}
\begin{split}
\kappa_k:C(\overline{G})&\rightarrow C_0(\mathcal{G})\\
f                       &\mapsto f\circ i_k^{-1}.p_{i_k(\overline{G})}
\end{split}
\end{equation*}
(where the $i$'s are as defined before Remark \ref{obs12}). These $*$-homomorphisms satisfy $\kappa_{l+m}\circ\widetilde{\psi}_{m,l+m}=\kappa_m$, since
\begin{equation*}
\begin{split}
\kappa_{l+m}\circ\widetilde{\psi}_{m,l+m}(f)(x)&=\widetilde{\psi}_{m,l+m}(f)\circ i_{l+m}^{-1}(x)p_{i_{l+m}(\overline{G})}(x)\\
&=f(i_{l+m}^{-1}(\varphi^{-l}(x)))p_{\varphi^l(\overline{G})}(x)p_{i_{l+m}(\overline{G})}(x)\\
&=f(i_{m}^{-1}(x))p_{i_m(\overline{G})}(x)\\
&=\kappa_m(f)(x).
\end{split}
\end{equation*}

Hence we have a $*$-homomorphism
$$
\D\lim_{\rightarrow}\{C(\overline{G}):\widetilde{\psi}_{m,l+m}\}\rightarrow C_0(\mathcal{G}).
$$

This is injective as each $\kappa_k$ is, because of $\kappa_k(f)\circ i_k=f$. It is also surjective as $\mathcal{G}=\overline{\D\cup_{m\in\mathds{N}^*}i_m(\overline{G})}$ and using the Stone-Weierstrass Theorem. So we have
$$
C^*(P)^s\cong C_0(\mathcal{G}).
$$
\end{proof}

Now we have all the tools to describe our C$^*$-algebra as a semigroup crossed product using $S=G\rtimes_\varphi\mathds{N}$. Consider the action
\begin{equation*}
\begin{split}
\alpha: S   &\rightarrow \hbox{End}(C^*(P))\\
       (g,n)&\mapsto u_gs^n (.) {s^*}^nu_{g^{-1}}.
\end{split}
\end{equation*}
\begin{theorem}\label{teo2}$\mathds{U}[\varphi]$ is isomorphic to $C^*(P)\rtimes_{\alpha} S$.
\end{theorem}
\begin{proof} By definition, $C^*(P)\rtimes_{\alpha} S$ together with
\begin{equation*}
\begin{split}
\iota_P: C^*(P)&\rightarrow C^*(P)\rtimes_{\alpha} S\\
         x     &\mapsto \iota_P(x)
\end{split}
\end{equation*}

and
\begin{equation*}
\begin{split}
\iota_S: S&\rightarrow \hbox{Isom}(C^*(P)\rtimes_{\alpha} S)\\
  (g,n)   &\mapsto \iota_S(g,n)
\end{split}
\end{equation*}

satisfying
$$
\iota_P(u_gs^nx{s^*}^nu_{g^{-1}})=\iota_S(g,n)\iota_P(x)\iota_S(g,n)^*
$$

is the crossed product of $(C^*(P),S,\alpha)$. But note that the triple $\mathds{U}[\varphi]$,
\begin{equation*}
\begin{split}
\pi: C^*(P)&\rightarrow \mathds{U}[\varphi]\\
         x &\mapsto x
\end{split}
\end{equation*}

and
\begin{equation*}
\begin{split}
\rho: S&\rightarrow \hbox{Isom}(\mathds{U}[\varphi])\\
  (g,n)&\mapsto u_gs^n
\end{split}
\end{equation*}

is a covariant representation of $(C^*(P),S,\alpha)$ because:
$$
\rho(g,n)\pi(x)\rho(g,n)^*=u_gs^nx{s^*}^nu_{g^{-1}}=\pi(\alpha_{(g,n)}(x)).
$$

Therefore there exists a $*$-homomorphism
\begin{equation}\label{teoiso1}
\Phi: C^*(P)\rtimes_{\alpha} S \rightarrow \mathds{U}[\varphi]
\end{equation}

such that $\Phi\circ\iota_P=\pi$ and $\Phi\circ\iota_S=\rho$.

In the other hand it is well known \cite{Laca2} that the crossed product $C^*(P)\rtimes_{\alpha} S$ is generated as a C$^*$-algebra by elements of the form $\iota_S(g,n)$ because we have$$\iota_P(u_gs^n{s^*}^nu_{g^{-1}})=\iota_S(g,n)\iota_S(g,n)^*.$$
But note that $\mathds{U}[\varphi]$ can be viewed as the universal C$^*$-algebra generated by the unitaries $\{u_g:\; g\in G\}$ and the isometry $s$ changing conditions (i) and (ii) in Definition \ref{defi1} to the equivalent one $u_gs^nu_hs^m=u_{g\varphi^n(h)}s^{n+m}$.

Therefore we identify $\iota_S(g,n)$ with $u_gs^n$ because the first ones satisfy the condition above, which generate $\mathds{U}[\varphi]$:
$$
\iota_S(g,n)\iota_S(h,m)=\iota_S(g\varphi^n(h),n+m)
$$

and
\begin{equation*}
\begin{split}
\D\sum_{g\in G/\varphi(G)}\iota_S(g,n)\iota_S(g,n)^*=&\D\sum_{g\in G/\varphi(G)}\iota_P(u_gs^n{s^*}^nu_{g^{-1}})\\
=&\iota_P\left(\D\sum_{g\in G/\varphi(G)}u_gs^n{s^*}^nu_{g^{-1}}\right)\\
=&\iota_P(1)=1.
\end{split}
\end{equation*}

Thus we get another $*$-homomorphism
\begin{equation}\label{teoiso2}
\begin{split}
\Delta: \mathds{U}[\varphi] &\rightarrow C^*(P)\rtimes_{\alpha} S\\
                      u_gs^n&\mapsto \iota_S(g,n).
\end{split}
\end{equation}

As (\ref{teoiso1}) and (\ref{teoiso2}) are inverses of each other we can conclude that
$\mathds{U}[\varphi]$ and $C^*(P)\rtimes_{\alpha} S$ are isomorphic.
\end{proof}

In order to be able to apply the exact sequence presented in \cite{Khoska} we split the action of $S$ presented above: we show that its semigroup crossed product is isomorphic to a semigroup crossed product implemented by $\mathds{N}$, where $\mathds{N}$ acts on a group crossed product by $G$.
\begin{proposition}\label{Prop1GN}The C$^*$-algebra $\mathds{U}[\varphi]$ is also isomorphic to the semigroup crossed product $(C^*(P)\rtimes_{\omega}G)\rtimes_{\tau}\mathds{N}$, where:
\begin{equation*}
\begin{split}
\omega: G &\rightarrow \hbox{Aut}(C^*(P))\\
        g &\mapsto u_g(\cdot)u_{g^{-1}}\\[2\baselineskip]
\tau: \mathds{N} &\rightarrow \hbox{End}(C^*(P)\rtimes_{\omega}G)\\
               n &\mapsto s^n(\cdot){s^*}^n
\end{split}
\end{equation*}

such that for $a_g\delta_g\in C^*(P)\rtimes_{\omega}G$, $\tau_n(a_g\delta_g)=s^na_g{s^*}^n\delta_{\varphi^n(g)}$.
\end{proposition}
\begin{proof} We will show that $C^*(P)\rtimes_{\alpha} S$ and $(C^*(P)\rtimes_{\omega}G)\rtimes_{\tau}\mathds{N}$ are isomorphic, by exploiting the universality of the semigroup crossed products, using two steps analogous to the first part of the proof of theorem above. Consider $C^*(P)\rtimes_{\alpha} S$ together with
\begin{equation*}
\begin{split}
\iota_P: C^*(P)&\rightarrow C^*(P)\rtimes_{\alpha} S\\
         x&\mapsto \iota_P(x)
\end{split}
\end{equation*}

and
\begin{equation*}
\begin{split}
\iota_S: S&\rightarrow \hbox{Isom}(C^*(P)\rtimes_{\alpha} S)\\
  (g,n)&\mapsto \iota_S(g,n)
\end{split}
\end{equation*}

satisfying
$$
\iota_P(u_gs^nx{s^*}^nu_{g^{-1}})=\iota_S(g,n)\iota_P(x)\iota_S(g,n)^*
$$

being the crossed product of $(C^*(P),S,\alpha)$. Analogously take $(C^*(P)\rtimes_{\omega}G)\rtimes_{\tau}\mathds{N}$ with
\begin{equation*}
\begin{split}
\iota_G: C^*(P)\rtimes_{\omega}G&\rightarrow (C^*(P)\rtimes_{\omega}G)\rtimes_{\tau}\mathds{N}\\
         a\delta_g&\mapsto \iota_G(a\delta_g)
\end{split}
\end{equation*}

and
\begin{equation*}
\begin{split}
\iota_N: \mathds{N}&\rightarrow \hbox{Isom}((C^*(P)\rtimes_{\omega}G)\rtimes_{\tau}\mathds{N})\\
  n&\mapsto \iota_N(n)
\end{split}
\end{equation*}

satisfying
$$
\iota_B(s^na\delta_g{s^*}^n\delta_{\varphi^n(g)})=\iota_N(n)\iota_B(a\delta_g)\iota_N(n)^*
$$

as the crossed product of $(C^*(P)\rtimes_{\omega}G,\mathds{N},\tau)$, where $a\delta_g$ represents the generating elements of $C^*(P)\rtimes_{\omega}G$, $g\in G$.

Note that the triple $(C^*(P)\rtimes_{\omega}G)\rtimes_{\tau}\mathds{N}$,
\begin{equation*}
\begin{split}
\varrho: C^*(P)&\rightarrow (C^*(P)\rtimes_{\omega}G)\rtimes_{\tau}\mathds{N}\\
         a&\mapsto \iota_G(a\delta_e)
\end{split}
\end{equation*}

and
\begin{equation*}
\begin{split}
\sigma: S&\rightarrow \hbox{Isom}((C^*(P)\rtimes_{\omega}G)\rtimes_{\tau}\mathds{N})\\
  (g,n)&\mapsto \iota_G(1\delta_g)\iota_N(n)
\end{split}
\end{equation*}

is a covariant representation of $(C^*(P),S,\alpha)$:
\begin{equation*}
\begin{split}
\sigma(g,n)\varrho(a)\sigma(g,n)^*&=\iota_G(1\delta_g)\iota_N(n)\iota_G(a\delta_e)\iota_N(n)^*\iota_G(1\delta_g)^*\\
&=\iota_G(1\delta_g)\iota_G(s^na{s^*}^n\delta_e)\iota_G(1\delta_g)^*\\
&=\iota_G(u_gs^na{s^*}^nu_{g^{-1}}\delta_g)\iota_G(1\delta_{g^{-1}})\\
&=\iota_G(u_gs^na{s^*}^nu_{g^{-1}}\delta_e)\\
&=\varrho(u_gs^na{s^*}^nu_{g^{-1}}).
\end{split}
\end{equation*}

Therefore we get a $*$-homomorphism
\begin{equation}\label{teo12iso1}
\Phi: C^*(P)\rtimes_{\alpha}S\rightarrow (C^*(P)\rtimes_{\omega}G)\rtimes_{\tau}\mathds{N}
\end{equation}

such that $\Phi\circ\iota_P=\varrho$ and $\Phi\circ\iota_S=\sigma$.

Let us find an inverse for $\Phi$ using the fact that the triple $C^*(P)\rtimes_{\alpha} S$,
\begin{equation*}
\begin{split}
\varpi: C^*(P)\rtimes_{\omega}G&\rightarrow C^*(P)\rtimes_{\alpha} S\\
         a\delta_g&\mapsto \iota_P(a)\iota_S(g,0)
\end{split}
\end{equation*}

and
\begin{equation*}
\begin{split}
\vartheta: \mathds{N}&\rightarrow \hbox{Isom}(C^*(P)\rtimes_{\alpha} S)\\
  n&\mapsto \iota_S(e,n)
\end{split}
\end{equation*}

is a covariant representation of $(C^*(P)\rtimes_{\omega}G,\mathds{N},\tau)$:
\begin{equation*}
\begin{split}
\vartheta(n)\varpi(a\delta_g)\vartheta(n)^*&=\iota_S(e,n)\iota_P(a)\iota_S(g,0)\iota_S(e,n)^*\\
&=\iota_S(e,n)\iota_S(g,0)\iota_P(u_{g^{-1}}au_g)\iota_S(e,n)^*\\
&=\iota_S(\varphi^n(g),0)\iota_S(e,n)\iota_P(u_{g^{-1}}au_g)\iota_S(e,n)^*\\
&=\iota_S(\varphi^n(g),0)\iota_P(s^nu_{g^{-1}a{s^*}^n}u_g)\\
&=\iota_S(\varphi^n(g),0)\iota_P(u_{\varphi^n(g^{-1})}s^na{s^*}^nu_{\varphi^n(g)})\\
&=\iota_P(s^na{s^*}^n)\iota_S(\varphi^n(g),0)=\varpi(s^na{s^*}^n\delta_{\varphi^n(g)}).
\end{split}
\end{equation*}

This implies the existence of a $*$-homomorphism
\begin{equation}\label{teo12iso2}
\Delta: (C^*(P)\rtimes_{\omega}G)\rtimes_{\tau}\mathds{N} \rightarrow C^*(P)\rtimes_{\alpha} S
\end{equation}

satisfying $\Delta\circ\iota_G=\varpi$ and $\Delta\circ\iota_N=\vartheta$.

Straightforward calculations show that the $*$-homomorphisms (\ref{teo12iso1}) and (\ref{teo12iso2}) are inverses of each other.
\end{proof}

\begin{example}\em\label{ex2} For any finite group $G$, an injective endomorphism will be surjective and therefore the isometry $s$ defining $\mathds{U}[\varphi]$ will be a unitary (by item (iii) of Definition \ref{defi1}). Then as $C^*(P)=\mathds{C}$,
$$
\mathds{U}[\varphi]\cong C^*(G)\rtimes_\tau\mathds{N}
$$

where
$$
\tau: \mathds{N}\rightarrow \hbox{End}(C^*(G))
$$

with
$$
\tau_n(\lambda u_g)=\lambda u_{\varphi^n(g)}.
$$

If one has the description of the K-theory of $C^*(G)$ it is easy to calculate the K-groups of $\mathds{U}[\varphi]$ by applying the Khoshkam-Skandalis sequence (\cite{Khoska}).
\end{example}
\begin{flushright}

  $\square$

  \end{flushright}

Since more results are known for group crossed products than for semigroup ones it is useful to find such a description of our C$^*$-algebra. We can do this using the minimal automorphic dilation of the semigroup crossed product system above (for more details, see Section 2 in \cite{Laca1}). One important requirement to use this dilation is that the semigroup must be an Ore semigroup: an Ore semigroup is a cancellative semigroup which is right-reversible i.e, it satisfies $Ss\cap Sr\neq \emptyset$ for all $s,r\in S$.
\begin{proposition}\label{propOre} The semidirect product $S=G\rtimes_\varphi\mathds{N}$ is an Ore semigroup.
\end{proposition}
\begin{proof} Consider $(g_i,n_i)\in S$ for $i\in\{1,2,3\}$. $S$ is cancellative:
\begin{equation*}
\begin{split}
&(g_1,n_1)(g_3,n_3)=(g_2,n_2)(g_3,n_3)\\
\Rightarrow\;&(g_1\varphi^{n_1}(g_3),n_1+n_3)=(g_2\varphi^{n_2}(g_3),n_2+n_3)\\
\Rightarrow\; &n_1=n_2 \hbox{ and }g_1\varphi^{n_1}(g_3)=g_2\varphi^{n_1}(g_3)\\
\Rightarrow\; &g_1=g_2
\end{split}
\end{equation*}

\begin{equation*}
\begin{split}
&(g_1,n_1)(g_2,n_2)=(g_1,n_1)(g_3,n_3)\\
\Rightarrow\; &(g_1\varphi^{n_1}(g_2),n_1+n_2)=(g_1\varphi^{n_1}(g_3),n_1+n_3)\\
\Rightarrow\; &n_2=n_3 \hbox{ and }\varphi^{n_1}(g_2)=\varphi^{n_1}(g_3)\\
\Rightarrow\; &g_2=g_3\hbox{ as }\varphi\hbox{ is injective}.
\end{split}
\end{equation*}

Also any two principal left ideals of $S$ intersect:
\begin{equation*}
\begin{split}
(\varphi^{n_2}(g_1^{-1}),n_2)(g_1,n_1)&=(e,n_2+n_1)\\
&=(\varphi^{n_1}(g_2^{-1}),n_1)(g_2,n_2)\in S(g_1,n_1)\cap S(g_2,n_2).
\end{split}
\end{equation*}
\end{proof}

It follows that the semigroup $S$ can be embedded in a group, called the enveloping group of $S$, which we will denote as $env(S)$, such that $S^{-1}S=env(S)$ (Theorem 1.1.2 \cite{Laca1}). It also implies that $S$ is a directed set by the relation defined by $(g,n)< (h,m)$ if $(h,m)\in S(g,n)$.  Let us define a candidate for $env(S)$. Consider
$$
\mathds{G}:=\D\lim_{\rightarrow}\{G_n:\varphi^n\}
$$

(with $G_n=G$ for all $n\in\mathds{N}$) and with the extended endomorphism $\overline{\varphi}$ construct the group
$$
\overline{S}:=\mathds{G}\rtimes_{\overline{\varphi}}\mathds{Z}.
$$

Then we can define an extended action $\overline{\alpha}$ of $\overline{S}$ over $C^*(P)^s$:
\begin{equation*}
\begin{split}
\overline{\alpha}:\overline{S}&\rightarrow \hbox{Aut}(C^*(P)^s)\\
                       (g_j,n)&\mapsto {s^*}^{j}u_gs^{n+j}(\cdot){s^*}^{n+j}u_{g^{-1}}s^{j}
\end{split}
\end{equation*}

(note that we can also find $g_j$ such that $j\geq |\,n|$).

Moreover, consider $i: C^*(P)\rightarrow C^*(P)^s$ the canonical inclusion.
\begin{proposition}\label{proporeG}The C$^*$-dynamical system $(C^*(P)^s,\overline{S},\overline{\alpha})$ is the minimal automorphic dilation of $(C^*(P),S,\alpha)$.
\end{proposition}
\begin{proof} Since the subset of $S$ containing all elements of the type $(e,n)$ is cofinal in $S$, we need only prove that $\overline{S}=env(S)$ (to use Theorem 2.1.1 in \cite{Laca1}). For this we need to show that $S$ is a subsemigroup of $\overline{S}$ and $\overline{S}\subset S^{-1}S$ \cite{CliPre}.

First it is obvious that $S$ is a subsemigroup of the group $\overline{S}$ via the inclusion $(g,n)\mapsto (g_0,n)$, where $g_0=g\in G=G_0\hookrightarrow\overline{G}$.

Without loss of generality take $(g_i,j)\in\overline{S}$ with $i>|j|$. Then
$$
(g_i,j)=(g_i,-i)(e,j+i)=(g_0,i)^{-1}(e,j+i)\in S^{-1}S.
$$
\end{proof}

We may conclude that the following theorem holds (Theorem 2.2.1 of \cite{Laca1}).

\begin{theorem}\label{teo22}The C$^*$-algebra $\mathds{U}[\varphi]$ is also isomorphic to the full corner $$\iota(1)(C^*(P)^s\rtimes_{\overline{\alpha}}\overline{S})\iota(1)\footnote{The isomorphism C$^*(P)\cong C(\overline{G})$ implemented in Proposition \ref{prop3} implies that the projection $\iota(1)\in C^*(P)^s$ corresponds to $p_{\overline{G}}\in C(\overline{G})$ viewed inside $C_0(\mathcal{G})$ via $i_0$ (defined before Remark \ref{obs12}).}.$$
\end{theorem}
\begin{flushright}

  $\square$

  \end{flushright}

Let us denote the isomorphism given by the last theorem by
$$
\beta: \mathds{U}[\varphi]\rightarrow \iota(1)(C^*(P)^s\rtimes_{\overline{\alpha}}\overline{S})\iota(1),
$$

and by Theorem 2.2.1 in \cite{Laca1} we know that
$$
\beta({s^*}^nu_{h^{-1}}fu_{h'}s^m)=i(1)U^*_{(h,n)}i(f)U_{(h',m)}i(1).
$$

Note that the isomorphism above implies that $\mathds{U}[\varphi]$ and $C^*(P)^s\rtimes_{\overline{\alpha}}\overline{S}$ are Morita equivalent and so they have the same K-groups.

To finish our identifications:
\begin{theorem}\label{teo3} The stabilization (Definition \ref{defi2}) $\mathds{U}[\varphi]^s$ is isomorphic to the group crossed product $C^*(P)^s\rtimes_{\overline{\alpha}}\overline{S}$.
\end{theorem}
\begin{proof} As in Theorem 2.4 in \cite{Laca1} we know that $\beta(u_g)=V(g_0,0)\iota(1)$ and $\beta(s^n)=V(e_0,n)\iota(1)$, where $V$ represents $\overline{S}$ in the crossed product $C^*(P)^s\rtimes_{\overline{\alpha}}\overline{S}$.

Define
\begin{equation*}
\begin{split}
\widetilde{\gamma}_{m,l+m}:\iota(1)(C^*(P)^s\rtimes_{\overline{\alpha}}\overline{S})\iota(1)&\rightarrow \iota(1)(C^*(P)^s\rtimes_{\overline{\alpha}}\overline{S})\iota(1)\\
x  &\mapsto V(e,l)xV(e,l)^*.
\end{split}
\end{equation*}

Remembering from Definition \ref{defi2} that
\begin{equation*}
\begin{split}
\psi_{m,l+m}:\mathds{U}[\varphi]&\rightarrow\mathds{U}[\varphi]\\
															 x&\mapsto s^lx{s^*}^l,
\end{split}
\end{equation*}

we can conclude that
$$
\beta\circ\psi_{m,l+m}\circ\beta^{-1}= \widetilde{\gamma}_{m,l+m}
$$

which implies the existence of an isomorphism
$$
\overline{\beta}:\mathds{U}[\varphi]^s\rightarrow\D\lim_{\rightarrow}\{\iota(1)(C^*(P)^s
\rtimes_{\overline{\alpha}}\overline{S})\iota(1),\;\widetilde{\gamma}_{m,l+m}\}.
$$

Moreover for $k\geq 0$ set
\begin{equation*}
\begin{split}
\lambda_k:\iota(1)(C^*(P)^s\rtimes_{\overline{\alpha}}\overline{S})\iota(1)&\rightarrow C^*(P)^s\rtimes_{\overline{\alpha}}\overline{S}\\
z                                                                          &\mapsto V^*(e,k)zV(e,k).
\end{split}
\end{equation*}

As
\begin{equation*}
\begin{split}
\lambda_{l+m}\circ\widetilde{\gamma}_{m,l+m}(z)&=V^*(e,l+m)V(e,l)zV(e,l)^*V(e,l+m)\\
&=V^*(e,m)zV(e,m)=\lambda_m(z),\\
\end{split}
\end{equation*}

we have a $*$-homomorphism
$$
\lambda:\D\lim_{\rightarrow}\{\iota(1)(C^*(P)^s\rtimes_{\overline{\alpha}}\overline{S})\iota(1):\widetilde{\gamma}_{m,l+m}\} \rightarrow C^*(P)^s\rtimes_{\overline{\alpha}}\overline{S}.
$$

It is injective because each $\lambda_k$ is. Moreover as
$$
\lambda_k(\iota(1))=\overline{\alpha}_{(e,-k)}(\iota(1))V_{(e,0)}={s^*}^k\iota(1)s^kV_{(e,0)}
$$

is an approximate unit for $C^*(P)^s\rtimes_{\overline{\alpha}}\overline{S}$, for $z\in C^*(P)^s\rtimes_{\overline{\alpha}}\overline{S}$ we have
\begin{equation*}
\begin{split}
&\D\lim_k\lambda_k(\iota(1)V(e,k)zV^*(e,k)\iota(1)\iota(1))\\
=&\D\lim_k[\lambda_k(\iota(1)V(e,k)zV^*(e,k)\iota(1))\lambda_k(\iota(1))]\\
=&\D\lim_k[V^*(e,k)\iota(1)V(e,k)zV^*(e,k)\iota(1)V(e,k)][{s^*}^k\iota(1)s^kV_{(e,0)}]\\
=&\D\lim_k{s^*}^k\iota(1)s^kV_{(e,0)}z{s^*}^k\iota(1)s^kV_{(e,0)}{s^*}^k\iota(1)s^kV_{(e,0)}\\
=&z,
\end{split}
\end{equation*}

and so $\lambda$ is surjective.

Consequently $
\mathds{U}[\varphi]^s\cong C^*(P)^s\rtimes_{\overline{\alpha}}\overline{S}.
$
\end{proof}

\begin{example}\em Consider a surjective endomorphism $\varphi$ of a group $G$. The surjectivity of $\varphi$ implies that $s$ is an isometry (by item (iii) of Definition \ref{defi1}). Moreover Proposition \ref{Prop1GN} together with the fact that $C^*(P)=\mathds{C}$\, implies that
$$
\mathds{U}[\varphi]\cong C^*(G)\rtimes_\tau\mathds{N},
$$

where
$$
\tau: \mathds{N} \rightarrow \hbox{End}(C^*(G))
$$

is defined by
$$
\tau_n(u_g)=u_{\varphi^n(g)}.
$$

Using the six-term exact sequence introduced by Khoshkam and Skandalis in \cite{Khoska}, one can build the sequence
$$
\begin{array}{ccccc}
 K_0(C^*(G))              &\xrightarrow{1-K_0(\tau_1)} &K_0(C^*(G)) &\rightarrow              & K_0(\mathds{U}[\varphi]) \\
  \uparrow                &                          &            &                         & \downarrow \\
 K_1(\mathds{U}[\varphi]) &\leftarrow                &K_1(C^*(G)) &\xleftarrow{1-K_1(\tau_1)} & K_1(C^*(G)) \\
\end{array}
$$

(note that this example is very similar to Example \ref{ex2}).
\end{example}
\begin{flushright}

  $\square$

  \end{flushright}

\section{Properties}

The crossed product description in last section implies two nice properties of $\mathds{U}[\varphi]$.

\begin{proposition}\label{propl14}If $G$ is amenable then $\mathds{U}[\varphi]$ is nuclear.
\end{proposition}
\begin{proof} $G$ being amenable implies that $\overline{S}$ is amenable as well (amenability is closed under direct limits by \cite{vNeu} and also closed under semidirect products). But we know that $C^*(P)^s$ is nuclear because it is commutative, therefore $C^*(P)^s\rtimes_{\overline{\alpha}}\overline{S}$ is nuclear by Proposition 2.1.2 in \cite{Ror}. Since hereditary C*-subalgebras of nuclear C*-algebras are nuclear by Corollary 3.3 (4) in \cite{ChoiEffros}, we conclude that
$$
\mathds{U}[\varphi]\cong C^*(P)\rtimes_{\alpha}S\cong i(1)(C^*(P)^s\rtimes_{\overline{\alpha}}\overline{S})i(1)
$$

is nuclear.
\end{proof}

\begin{proposition}\label{propl15}If $G$ is amenable then $\mathds{U}[\varphi]$ satisfies the UCT property.
\end{proposition}
\begin{proof} Since $C^*(P)^s$ is commutative, $C^*(P)^s\rtimes_{\overline{\alpha}}\overline{S}$ is isomorphic to a groupoid C$^*$-algebra. When the group $G$ is amenable then $\overline{S}$ also is, and the respective groupoid is also amenable. Therefore using a result by Tu (\cite{tutu} Proposition 10.7), the crossed product satisfies UCT. By Morita equivalence, $\mathds{U}[\varphi]$ also satisfies it.
\end{proof}

We will now prove that our algebra $\mathds{U}[\varphi]$ is purely infinite and simple. We will proceed in the same way as in \cite{Cuntz2} and in many other papers: we present a particular faithful conditional expectation and a dense $*$-subalgebra of $\mathds{U}[\varphi]$ such that the conditional expectation of any positive element of this $*$-subalgebra can be described using a finite number of pairwise orthogonal projections.

For this purpose we will use the description in Theorem \ref{teo22} of $\mathds{U}[\varphi]$ as a corner of a group crossed product. To define the conditional expectation, we require the amenability of the group $G$: therefore the group $\overline{S}=\mathds{G}\rtimes_{\overline{\varphi}}\mathds{Z}$ (defined after Proposition \ref{propOre}) is also amenable (as mentioned in the proof of Proposition \ref{propl14}). This condition is necessary because we want to use the well-known result which says that there exists a canonical faithful conditional expectation on the reduced group crossed product, and the amenability of $\overline{S}$ implies that both the full and the reduced group crossed products (implemented by $\overline{S}$-actions) are isomorphic.

The main tool of this section is the following (proven in Proposition 5.2 of \cite{Li1}).
\begin{proposition}\label{prop13}Let $\widetilde{A}$ be a dense $*$-subalgebra of a unital C$^*$-algebra $A$. Assume that $\epsilon$ is a faithful conditional expectation on $A$ such that for every $0\neq x\in\widetilde{A}_{+}$ there exist finitely many projections $f_i\in A$ with
\begin{itemize}
  \item[(i)] $f_i\bot f_j$,$\forall\; i\neq j$;
  \item[(ii)] $f_i\sim_{s_i} 1$, via\footnote{I.e.: $\exists\; s_i$ isometries such that $s_is_i^*=f_i$,$\forall\; i$} isometries $s_i\in A$, $\forall\; i$;
  \item[(iii)] $\left\|\D\sum_if_i\epsilon(x)f_i\right\|=\|\epsilon(x)\|$;
  \item[(iv)] $f_ixf_i=f_i\epsilon(x)f_i\in\mathds{C}f_i$,$\forall\; i$.
\end{itemize}
Then $A$ is purely infinite and simple.
\end{proposition}
\begin{flushright}

  $\square$

  \end{flushright}

Moreover in order to find these projections it is also necessary to require that the injective endomorphism $\varphi$ is pure, i.e:
$$
\D\bigcap_{n\in\mathds{N}}\varphi^n(G)=\{e\}.
$$

In order to apply the proposition above the first step is to define a conditional expectation. As mentioned before, we require that the group $G$ is amenable. Remember the isomorphism from Theorem \ref{teo22}:
$$
\beta: \mathds{U}[\varphi]\rightarrow \iota(1)(C^*(P)^s\rtimes_{\overline{\alpha}}\overline{S})\iota(1).
$$

\begin{proposition}\label{prop1}There exists a faithful conditional expectation
\begin{equation*}
\begin{split}
\epsilon: \mathds{U}[\varphi]&\rightarrow \beta^{-1}(\iota(1)C^*(P)^s\iota(1))\\
{s^n}^*u_{h^{-1}}fu_{h'}s^m&\mapsto\left\{
                                      \begin{array}{ll}
                                        {s^n}^*u_{h^{-1}}fu_hs^n, & \hbox{if }n=m\hbox{ and }h=h'; \\
                                        0, & \hbox{otherwise.}
                                      \end{array}
                                    \right.
\end{split}
\end{equation*}

for all $h$, $h'\in G$ and $n$, $m\in\mathds{N}$.
\end{proposition}
\begin{proof} As $\overline{S}$ is amenable, the isomorphism $\beta$ of Theorem \ref{teo22} can be expanded to include also the reduced group crossed product
$$
\mathds{U}[\varphi]\cong \iota(1)(C^*(P)^s\rtimes_{\overline{\alpha}}\overline{S})\iota(1)\cong
\iota(1)(C^*(P)^s\rtimes_{r,\overline{\alpha}}\overline{S})\iota(1).
$$

Let us denote the elements of $\overline{S}$ by $s$ and its identity by $e$. We will also use $\delta_s$ to denote the unitary elements implementing the action of $\overline{S}$ in the crossed product. Consider the well-known faithful conditional expectation on the reduced group crossed product:
\begin{equation*}
\begin{split}
E: C^*(P)^s\rtimes_{r,\overline{\alpha}}\overline{S}&\rightarrow C^*(P)^s\\
                                           x\delta_s&\mapsto \left\{
                                                               \begin{array}{ll}
                                                                 x, & \hbox{if }s=e; \\
                                                                 0, & \hbox{otherwise.}
                                                               \end{array}
                                                             \right.
\end{split}
\end{equation*}

Straightforward calculations show that the following is also a faithful conditional expectation:
\begin{equation*}
\begin{split}
\overline{E}: \iota(1)(C^*(P)^s\rtimes_{r,\overline{\alpha}}\overline{S})\iota(1)&\rightarrow \iota(1)C^*(P)^s\iota(1)\\
                                                        \iota(1)x\delta_s\iota(1)&\mapsto \left\{
                                                               \begin{array}{ll}
                                                                 \iota(1)x\iota(1), & \hbox{if }s=e; \\
                                                                         0, & \hbox{otherwise.}
                                                               \end{array}
                                                             \right.
\end{split}
\end{equation*}

Using the isomorphism $\beta$ we can rewrite $\overline{E}$ to conclude that we have the faithful conditional expectation
\begin{equation*}
\begin{split}
\epsilon: \mathds{U}[\varphi]&\rightarrow \beta^{-1}(\iota(1)C^*(P)^s\iota(1))\\
{s^n}^*u_{h^{-1}}fu_{h'}s^m&\mapsto\left\{
                                      \begin{array}{ll}
                                        {s^n}^*u_{h^{-1}}fu_hs^n, & \hbox{if }n=m\hbox{ and }h=h'; \\
                                        0, & \hbox{otherwise.}
                                      \end{array}
                                    \right.
\end{split}
\end{equation*}
\end{proof}

Now, to find projections to describe the image of $y\in span(Q)_{+}$ under the conditional expectation $\epsilon$ presented above, remember that $y$ has the form
$$
y=\D\sum_{m,n,h,h',f}a_{(m,n,h,h',f)}{s^*}^nu_{h^{-1}}fu_{h'}s^m
$$

for $m$, $n\in\mathds{N}$, $h$, $h'\in G$, $f\in P$ and $a_{(\ldots)}\neq 0$. As we have finitely many projections of $C^*(P)$ in the description of $y$, write them all as sums of (altogether $N$) mutually orthogonal projections $u_{g_i}s^M{s^*}^Mu_{g_i^{-1}}$, with $g_i\in G/\varphi^M(G)$, for all $1\leq i\leq N$ and $M\in\mathds{N}$ big enough.
\begin{proposition}\label{prop2}There are $N$ pairwise orthogonal projections $f_1,\ldots f_N\in P$ such that
\begin{enumerate}
  \item[(i)] $\Phi$ defined by
\begin{equation*}
\begin{split}
C^*(\{u_{g_1}s^M{s^*}^Mu_{g_1^{-1}},\ldots,u_{g_N}s^M{s^*}^Mu_{g_N^{-1}}\})&\rightarrow C^*(\{f_1,\ldots,f_N\})\\
z&\mapsto\D\sum_{i=1}^Nf_izf_i
\end{split}
\end{equation*}

is an isomorphism, and
  \item[(ii)] $\Phi(\epsilon(y))=\D\sum_{i=1}^Nf_iyf_i$, $\forall\; y\in\mathds{U}[\varphi]$.
\end{enumerate}
\end{proposition}
\begin{proof} Define
$$
f_i:=u_{h_i}s^p{s^*}^pu_{h_i^{-1}}
$$

for some $p\in\mathds{N}$ bigger than $M$ (in fact, we may choose $p$ as big as we want), where $g_i^{-1}h_i\in\varphi^M(G)$. This implies that the set of the $f_i$'s is orthogonal and that (i) holds.

For (ii), first note that when $\delta_{m,n}\delta_{h,h'}=1$ it is true that $\epsilon=Id$, and so (ii) is satisfied. So let us take a look on those summands in $y$ with $\delta_{m,n}\delta_{h,h'}=0$ (we will say that such an element has \emph{critical index} $(m,n,h,h',f)$). The conditional expectation $\epsilon$ maps these summands to 0 and in order for (ii) to be satisfied we need that, for all $1\leq i\leq N$,
$$
f_i{s^*}^nu_{h^{-1}}fu_{h'}s^mf_i=0.
$$

We calculate
\begin{equation*}
\begin{split}
&f_i{s^*}^nu_{h^{-1}}fu_{h'}s^mf_i\\
&={s^*}^nu_{h^{-1}}(u_hs^nf_i{s^*}^nu_{h^{-1}})f(u_{h'}s^mf_i{s^*}^mu_{h'^{-1}})u_{h'}s^m\\
&={s^*}^nu_{h^{-1}}[u_{h\varphi^n(h_i)}s^{n+p}{s^*}^{n+p}u_{\varphi^n(h_i^{-1})h^{-1}}\\ &u_{h'\varphi^m(h_i)}s^{m+p}{s^*}^{n+p}
u_{\varphi^m(h_i^{-1})h'^{-1}}]fu_{h'}s^m.
\end{split}
\end{equation*}

Now, analysing only the expression between the brackets,
\begin{equation*}
\begin{split}
&[u_{h\varphi^n(h_i)}s^{n+p}{s^*}^{n+p}u_{\varphi^n(h_i^{-1})h^{-1}}u_{h'\varphi^m(h_i)}s^{m+p}{s^*}^{m+p}
u_{\varphi^m(h_i^{-1})h'^{-1}}]\\
&=\left(\D\sum_{g\in G/\varphi^m(G)}u_{h\varphi^n(h_i)\varphi^{n+p}(g)}s^{m+n+p}{s^*}^{m+n+p}u_{\varphi^{n+p}(g^{-1})\varphi^n(h_i^{-1})h^{-1}}\right)\\
&\times\left(\D\sum_{k\in G/\varphi^n(G)}u_{h'\varphi^m(h_i)\varphi^{m+p}(k)}s^{m+n+p}{s^*}^{m+n+p}u_{\varphi^{m+p}(k^{-1})\varphi^m(h_i^{-1})h^{-1}}\right).
\end{split}
\end{equation*}

This product will be zero if the two sums are mutually orthogonal, which happens if for all $g\in G/\varphi^m(G)$ and $k\in G/\varphi^n(G)$,
\begin{equation*}
h\varphi^n(h_i)\varphi^{n+p}(g)\varphi^{m+n+p}(x)\neq h'\varphi^m(h_i)\varphi^{m+p}(k)\varphi^{m+n+p}(y),\;\forall\; x,y\in G
\end{equation*}

which is equivalent to
\begin{equation*}
\varphi^{m+p}(k^{-1})\varphi^m(h_i^{-1})h'^{-1}h\varphi^n(h_i)\varphi^{n+p}(g)\neq\varphi^{m+n+p}(z),\;\forall\; z\in G.
\end{equation*}

A sufficient condition for this to hold is that $\varphi^m(h_i^{-1})h'^{-1}h\varphi^n(h_i)\neq\varphi^p(z)$ $\forall\; z\in G$, for each critical index $(m,n,h,h',f)$. Using the fact that $\varphi$ is pure we may choose some $p_{(m,n,h,h',f)}\in\mathds{N}$ such that
$$
\varphi^m(h_i^{-1})h'^{-1}h\varphi^n(h_i)\notin\varphi^{p_{(m,n,h,h',f)}}(G).
$$

As we have a finite number of critical indices, it is sufficient to take the
biggest $p_{(m,n,h,h',f)}$ and call it $p$.
\end{proof}

To understand this choice of $p$, consider the following example.

\begin{example}\em\label{ex1} Let $G=\mathds{Z}$ and
\begin{equation*}
\begin{split}
\varphi:\mathds{Z}&\rightarrow\mathds{Z}\\
n                 &\mapsto 3n.
\end{split}
\end{equation*}
Then we have $\frac{G}{\varphi(G)}=\{\overline{0},\overline{1},\overline{2}\}$, $\frac{G}{\varphi^2(G)}=\{\overline{0},\overline{1},\ldots,\overline{8}\}$ and in general
$$
\frac{G}{\varphi^n(G)}=\{\overline{0},\ldots,\overline{3^n-1}\}=\mathds{Z}_{3^n}.
$$

Take the following $y\in span(Q)$
\begin{equation*}
\begin{split}
y=&2{s^*}^2u_{30}(u_5s^{4}{s^*}^4u_{-5})u_{2187}s^{1}-4{s^*}^7u_0(u_{10}s^{4}{s^*}^4u_{-10})u_{-5}s^9\\ &+{s^*}^8u_{20}s^{4}{s^*}^4u_{-20}s^8
\end{split}
\end{equation*}

and note that in $y$ we have two terms with critical indices (the first ones).

Using the notation of the above proposition, $M=4$ and, for the first term of $y$: $n=2$, $m=1$, $h=-30$, $h'=2187$ and $g_1=5$. Choosing $h_1=86$, it is true that $-g_1+h_1=-5+86=81\in\varphi(G)$. Then:
$$
\varphi^{1}(-86)-2187-30+\varphi^{2}(86)=-1701=\varphi^5(7)\notin\varphi^6(\mathds{Z}).
$$

So $p_1:=p_{(1,2,-30,2187,f)}=6$ (or bigger). For the second term it is not hard to see that $p_2=1$:
$$
\varphi^{9}(-91)+5-0+\varphi^{7}(91)=-1592131\notin\varphi^1(\mathds{Z}).
$$

So one can choose any $p\geq 6$.
\end{example}
\begin{flushright}

  $\square$

  \end{flushright}

Using the description above of the faithful conditional expectation
$$
\epsilon: \mathds{U}[\varphi]\rightarrow \beta^{-1}(\iota(1)C^*(P)^s\iota(1))
$$

where $P=\{u_gs^n{s^*}^nu_{g^{-1}}:\;g\in G,\; n\in\mathds{N}\}$, together with the dense $*$-subalgebra
$$
\hbox{span}(Q)=\hbox{span}(\{{s^*}^nu_{h^{-1}}fu_{h'}s^m:\;f\in P, h, h'\in G, n,m\in\mathds{N}\}),
$$

we can prove the main result of this section by applying Propositions \ref{prop2} and \ref{prop13} (the definition of pure infiniteness comes from \cite{Cuntz2}).
\begin{theorem}\label{teo1}Let $G$ be a discrete countable amenable group and $\varphi$ a pure injective endomorphism of $G$ with finite cokernel. Then the C$^*$-algebra $\mathds{U}[\varphi]$ is simple and purely infinite, i.e. for all non zero $x\in\mathds{U}[\varphi]$ there are $a$, $b\in\mathds{U}[\varphi]$ with $axb=1$.
\end{theorem}
\begin{flushright}

  $\square$

  \end{flushright}

\begin{corollary} When satisfied the conditions of the theorem above, the universal C$^*$-algebra $\mathds{U}[\varphi]$ is isomorphic to $C_r^*[\varphi]$, as defined in Definitions \ref{defi1} and \ref{defi1red} respectively.
\end{corollary}
\begin{flushright}

  $\square$

  \end{flushright}

\begin{theorem} If the conditions of the theorem above are satisfied, the universal C$^*$-algebra $\mathds{U}[\varphi]$ is a Kirchberg algebra satisfying the UCT property.
\end{theorem}
\begin{flushright}

  $\square$

  \end{flushright}

It would be interesting to know if the conditions of the theorem above are also necessary: if we construct the C$^*$-algebra associated with some injective endomorphism of an amenable group, is it simple and purely infinite only if $\varphi$ is pure? Unfortunately we don't answer this question here, but the next trivial example gives some idea about this direction.
\begin{example}\em\label{ex0} For some commutative group $G$ (thus amenable), consider $\varphi = id_G$.

As $u_gs=su_g$ for all $g\in G$ ($\varphi$ is trivial), our C$^*$-algebra will be commutative. Now, as $\frac{G}{\varphi(G)}$ has only the element $\{e\}$, condition (iii) of Definition \ref{defi1} implies that the isometry $s$ is a unitary. Then $\mathds{U}[\varphi]$ is the commutative C$^*$-algebra generated by the unitaries $\{u_g, s:g\in G\}$, and this one is the non-simple tensor product $C^*(G)\otimes C^*(\mathds{Z})=C^*(G)\otimes C(\mathcal{S}^1)$.

Moreover using the K\"{u}nneth Formula \cite{Schoc} we conclude that
\begin{equation*}
K_0(\mathds{U}[\varphi])=K_1(\mathds{U}[\varphi])=K_0(C^*(G))\oplus K_1(C^*(G)).
\end{equation*}
\end{example}
\begin{flushright}

  $\square$

  \end{flushright}

\section{Description of $\mathds{U}[\varphi]$ via group partial crossed products}\label{partialcrpr}

In \cite{BoEx} Boava and Exel constructed a partial group algebra isomorphic to the C$^*$-algebra $\mathds{U}[R]$ associated with a integral domain $R$ \cite{Culi1}. Consequently due to Theorem 4.4 of \cite{ExLaQu} one can define a certain partial crossed product which is isomorphic to $\mathds{U}[R]$. With the latter description it is proven in \cite{BoEx}, using only tools from partial crossed products, that if $R$ is not a field then $\mathds{U}[R]$ is simple (which is part of the conclusion of Li \cite{Li1}, namely, Corollary 5.14).

In this section we will present analogous results adapted to our case, i.e., given a C$^*$-algebra $\mathds{U}[\varphi]$ associated with some injective endomorphism $\varphi$ of a group $G$ with unit $e$, we will show that $\mathds{U}[\varphi]$ can also be viewed as a partial group algebra and, consequently, as a partial crossed product. The ideas follow the ones presented in \cite{BoEx}.

With this description we show that when $G$ is amenable we can rewrite the faithful conditional expectation $\epsilon$ presented in Proposition \ref{prop1} in terms of the partial group crossed product. To finish we use a well known result from the theory of group partial crossed products to prove a weaker result than Theorem \ref{teo1}: if $G$ is commutative and $\varphi$ is pure then $\mathds{U}[\varphi]$ is simple.

We start with an introduction to partial actions, partial crossed products and partial group algebras, before presenting the right isomorphisms and descriptions of $\mathds{U}[\varphi]$.

\begin{definition}\label{pcpdefi1}A partial action $\alpha$ of a group $G$ on a C$^*$-algebra $A$ is a collection of closed two-sided ideals $\{D_g\}_{g\in G}$ of $A$ and $*$-isomorphisms $\alpha_g: D_{g^{-1}}\rightarrow D_g$ satisfying
\begin{itemize}
  \item[(PA1)] $D_e=A$;
  \item[(PA2)] $\alpha_h^{-1}(D_h\cap D_{g^{-1}})\subseteq D_{(gh)^{-1}}$;
  \item[(PA3)] $\alpha_g\circ\alpha_h(x)=\alpha_{gh}(x)$, $\forall\; x\in \alpha_h^{-1}(D_h\cap D_{g^{-1}})$.
\end{itemize}
\end{definition}

Using (PA1) - (PA3) one can show that $\alpha_e=$ id$_A$, $\alpha_{g^{-1}}=\alpha_g^{-1}$ and that\nl $\alpha_h^{-1}(D_h\cap D_{g^-1})=D_{(gh)^{-1}}\cap D_{h^{-1}}$.

Analogously, one can define a partial action of $G$ acting on a locally compact space $X$: just replace the ideals $D_g$ by open sets $X_g\subseteq X$ and the $*$-isomorphisms $\alpha_g$ by homeomorphisms $\theta_g: X_{g^{-1}}\rightarrow X_g$.

We call the triples $(\alpha,G,A)$ or $(\theta,G,X)$ partial dynamical systems, or partial actions when there is no possibility of misunderstanding.
\begin{example}\em\label{Expa1} If $\theta$ is a partial action of $G$ on the locally compact space $X$ with $\theta_g: X_{g^{-1}}\rightarrow X_g$, one can easily construct a partial action of $G$ on the C$^*$-algebra $C_0(X)$ considering $D_g=C_0(X_g)$ and
\begin{equation*}
\begin{split}
\alpha_g:C_0(X_{g^{-1}})&\rightarrow C_0(X_g)\\
                       f&\mapsto f\circ\theta_{g^{-1}}.
\end{split}
\end{equation*}
\end{example}
\begin{flushright}

  $\square$

  \end{flushright}

Now we want to define partial crossed products. There are three ways to realize them: one using Fell bundles (and we recommend \cite{ExelFell}), another using enveloping C$^*$-algebras (for details and some interesting examples look at Section 2 of \cite{Mc}) and the last one as a universal object with respect to covariant pairs (see Section 3 of \cite{QuRa}). We use the last way in our proofs and therefore we present it.

Let us define first a particular set of representations called partial representations.

\begin{definition}\label{pcpdefi2}A partial representation $\pi$ of a group $G$ into a unital C$^*$-algebra $B$ is a map $\pi: G\rightarrow B$ satisfying
\begin{itemize}
  \item[(PR1)] $\pi(e)=1$;
  \item[(PR2)] $\pi(g^{-1})=\pi(g)^*$;
  \item[(PR3)] $\pi(g)\pi(h)\pi(h^{-1})=\pi(gh)\pi(h^{-1})$.
\end{itemize}
\end{definition}

Then the partial group crossed product $A\rtimes_{\alpha}G$ is defined as the universal object with respect to a covariant pair $(\upsilon,\pi)$, which means a $*$-homomorphism ($B$ being a unital C$^*$-algebra)
$$
\upsilon: A\rightarrow B
$$

and a partial representation of $G$
$$
\pi: G\rightarrow B
$$

satisfying
\begin{equation*}
\begin{split}
\upsilon(\alpha_g(x))&=\pi(g)\upsilon(x)\pi(g^{-1})\hbox{ for }x\in D_{g^{-1}},\\
\upsilon(x)\pi(g)\pi(g^{-1})&=\pi(g)\pi(g^{-1})\upsilon(x)\hbox{ for }x\in A.
\end{split}
\end{equation*}

To define a partial group algebra, consider the set $[G]:=\{[g]:\;g\in G\}$ (without any operations).

\begin{definition}\label{defipagr}The partial group algebra of $G$, denoted $C^*_p(G)$, is the universal C$^*$-algebra generated by $[G]$ with respect to the relations
\begin{itemize}
  \item[(R$_p$1)] $[e]=1$;
  \item[(R$_p$2)] $[g^{-1}]=[g]^*$;
  \item[(R$_p$3)] $[g][h][h^{-1}]=[gh][h^{-1}]$.
\end{itemize}
\end{definition}

The C$^*$-algebra $C_p^*(G)$ is universal with respect to partial representations of $G$ (note the equivalence between relations (R$_p$) and (PR) of Definition \ref{pcpdefi2}).

In fact, one can define partial group algebras for more restricted situations, i.e., requiring that $[G]$ satisfies additional relations than the 3 relations above. Let us set $e_g:=[g][g^{-1}]$ and for our constructions consider $\mathcal{R}$ a set of (extra) relations on $[G]$ such that every relation is of the form
\begin{equation}\label{relationspga}
\sum_i\prod_je_{g_{ij}}=0.
\end{equation}

\begin{definition}The partial group algebra of $G$ with relations $\mathcal{R}$, denoted $C_p^*(G,\mathcal{R})$, is defined to be the universal C$^*$-algebra generated by $[G]$ with relations $R_p\cup\,\mathcal{R}$. This C$^*$-algebra is universal with respect to partial representations which satisfy $\mathcal{R}$.
\end{definition}

An interesting fact is that the class of partial group algebras without restrictions and of the ones with extra relations of the type (\ref{relationspga}) is contained in the class of partial crossed products (Definition 6.4 of \cite{Exel1} and Theorem 4.4 of \cite{ExLaQu} respectively). In our case the C$^*$-algebra $\mathds{U}[\varphi]$ will be isomorphic to a partial group algebra with additional relations of the form (\ref{relationspga}) above, and we will show how these can be viewed as partial crossed products.

Consider the power set $\mathcal{P}(G)$ (of $G$) with the topology given by identifying it with the compact set $\{0,1\}^G$, and denote $X_G$ the subset of $\mathcal{P}(G)$ of the subsets $\xi$ of $G$ which contain $e\in G$. Note that using the product topology of $\{0,1\}^G$ implies that $X_G$ is compact and Hausdorff.

Denote by $1_g$ the following function in $C(X_G)$:
$$
1_g(\xi)=\left\{
  \begin{array}{ll}
    1, & \hbox{if }g\in\xi; \\
    0, & \hbox{otherwise.}
  \end{array}
\right.
$$

Denote $\widehat{\mathcal{R}}$ the subset of $C(X_G)$ given by the functions $\sum_i\prod_j1_{g_{ij}}$ where the relation  $\sum_i\prod_je_{g_{ij}}=0$ is in $\mathcal{R}$. The \emph{spectrum} of the relations $\mathcal{R}$ is defined to be the compact (Proposition 4.1 \cite{ExLaQu}) space
$$
\Omega_\mathcal{R}:=\{\xi\in X_G:\; f(g^{-1}\xi)=0,\;\forall\; f\in\widehat{\mathcal{R}},\;\forall\; g\in\xi\}.
$$

Now for $g\in G$, consider
$$
\Omega_g:=\{\xi\in\Omega_\mathcal{R}:\; g\in\xi\}
$$

and let us define
\begin{equation*}
\begin{split}
\theta_g: \Omega_{g^{-1}}&\rightarrow\Omega_g\\
                      \xi&\mapsto g\xi.
\end{split}
\end{equation*}

Then we have defined a partial action $\theta$ of $G$ on $\Omega_\mathcal{R}$. Turning this partial action (as in Example \ref{Expa1}) into a partial action $\alpha$ of $G$ on $C(\Omega_\mathcal{R})$, it is well known (by Theorem 4.4 (iii) in \cite{ExLaQu}) that
\begin{equation}\label{parcpc14}
C_p^*(G,R)\cong C(\Omega_\mathcal{R})\rtimes_{\alpha}G.
\end{equation}

Now let us find a partial group C$^*$-algebra description of $\mathds{U}[\varphi]$. Therefore recall the set $\overline{S}=\mathds{G}\rtimes_{\overline{\varphi}}\mathds{Z}$ whose elements will be denoted by $(g_i,n)$ with $g_i\in G_i\subseteq\mathds{G}$. In case $g\in G=G_0\subseteq\mathds{G}$ we will use the notation $(g,n)$.

Consider the following relations $\mathcal{R}$:
\begin{itemize}
  \item[($\mathcal{R}_1$)] $[(g,0)][(g,0)^{-1}]=1,\;\forall\; g\in G$;
  \item[($\mathcal{R}_2$)] $[(e,-n)][(e,-n)^{-1}]=1\;\forall\; n\in\mathds{N}$;
  \item[($\mathcal{R}_3$)] $\D\sum_{g\in\frac{G}{\varphi^n(G)}}[(g,n)][(g,n)^{-1}]=1,\;\forall\; n\in\mathds{N}$.
\end{itemize}

Consider also the partial group algebra relations in this case i.e, on the group $\overline{S}$:
\begin{itemize}
  \item[(R$_p$1)] $[(e,0)]=1$;
  \item[(R$_p$2)] $[(g_i,n)^{-1}]=[(g_i,n)]^*,\;\forall\; n\in\mathds{Z},\;\forall\; g_i\in\mathds{G}$;
  \item[(R$_p$3)] $[(g_i,n)][(h_j,m)][(h_j,m)^{-1}]=[(g_i\varphi^n(h_j),n+m)][(h_j,m)^{-1}],\nl\;\forall\; m,n\in\mathds{Z},\;\forall\; g_i,h_j\in\mathds{G}$.
\end{itemize}

Define
\begin{equation*}
\begin{split}
\pi: \overline{S}&\rightarrow\mathds{U}[\varphi]\\
          (g_i,n)&\mapsto {s^*}^iu_gs^{n+i},
\end{split}
\end{equation*}

remembering that we can always suppose $i\geq |\,n|$. Note that when $g\in G$, $\pi(g,n)=u_gs^n$.

\begin{proposition}The map $\pi$ is a partial representation of $\overline{S}$ which satisfies the relations $\mathcal{R}$.
\end{proposition}
\begin{proof} First we prove that $\pi$ is a partial representation of $\overline{S}$.

(R$_p$1): $\pi((e,0))=u_e=1$;

(R$_p$2): \begin{equation*}
\begin{split}
\pi((g_i,n)^{-1})&=\pi((g^{-1}_{i+n},-n))={s^*}^{i+n}u_{g^{-1}}s^i=({s^*}^iu_gs^{i+n})^*\\
&=(\pi((g_i,n)))^*;
\end{split}
\end{equation*}

(R$_p$3): \begin{equation*}
\begin{split}
&\pi((\varphi^j(g)\overline{\varphi}^{i+n}(h)_{i+j},n+m))\pi((h_j,m)^{-1})\\
&={s^*}^{i+j}u_{\varphi^j(g)\varphi^{i+n}(h)}s^{i+j+n+m}{s^*}^{j+m}u_{h^{-1}}s^j\\
&={s^*}^iu_g{s^*}^js^{i+n}\underbrace{u_hs^{j+m}{s^*}^{j+m}u_{h^{-1}}}\underbrace{s^j{s^*}^j}s^j\\
&={s^*}^iu_g{s^*}^js^{i+n}s^j{s^*}^ju_hs^{j+m}{s^*}^{j+m}u_{h^{-1}}s^j\\
&={s^*}^iu_gs^{i+n}{s^*}^ju_hs^{j+m}{s^*}^{j+m}u_{h^{-1}}s^j\\
&=\pi((g_i,n))\pi((h_j,m))\pi((h_j,m)^{-1}).
\end{split}
\end{equation*}

Now we show that $\pi$ satisfies the extra relations $\mathcal{R}$.

($\mathcal{R}_1$): $\pi((g,0))\pi((g,0)^{-1})=u_e=1$;

($\mathcal{R}_2$): $\pi((e,-n))\pi((e,-n)^{-1})=\pi((e_n,-n))\pi((e_n,-n)^{-1})={s^*}^ns^n=1$;

($\mathcal{R}_3$): $\D\sum_{g\in\frac{G}{\varphi^n(G)}}\pi((g,n))\pi((g,n)^{-1})=\D\sum_{g\in\frac{G}{\varphi^n(G)}}u_gs^n{s^*}^{-n}u_{g^{-1}}=1$.
\end{proof}

It follows from the universality of the partial group algebra $C_p^*(\overline{S},\mathcal{R})$ that there exists a $*$-homomorphism
\begin{equation}\label{eqpga1}
\begin{split}
\Phi: C^*_p(\overline{S},\mathcal{R})&\rightarrow\mathds{U}[\varphi]\\
                            [(g_i,n)]&\mapsto {s^*}^iu_gs^{n+i}.
\end{split}
\end{equation}

Let us find an inverse for $\Phi$ by using the relations which define $\mathds{U}[\varphi]$.

\begin{proposition}The (obviously) unitary elements $[(g,0)]$ and isometries $[(e,n)]$ of $C_p^*(\overline{S},\mathcal{R})$ satisfy the relations which define $\mathds{U}[\varphi]$.
\end{proposition}
\begin{proof} Let us show that the elements above satisfy the relations (i) - (iii) of Definition \ref{defi1}.

(i):\begin{equation*}
\begin{split}
[(g,0)][(h,0)]&=[(g,0)][(h,0)][(h^{-1},0)][(h,0)]=[(gh,0)][(h^{-1},0)][(h,0)]\\
&=[(gh,0)];
\end{split}
\end{equation*}

(ii):\begin{equation*}
\begin{split}
[(e,1)][(g,0)]&=[(e,1)][(g,0)][(g^{-1},0)][(g,0)]=[(\varphi(g),1)][(g^{-1},0)][(g,0)]\\
&=[(\varphi(g),1)]=[(\varphi(g),1)][(e,-1)][(e,1)]\\
&=[(\varphi(g),0)][(e,1)][(e,-1)][(e,1)]\\
&=[(\varphi(g),0)][(e,1)];
\end{split}
\end{equation*}

(iii):
\begin{equation*}
\begin{split}
[(g,0)][(e,1)][(e,-1)][(g^{-1},0)]&=[(g,1)][(e,-1)][(g^{-1},0)]\\
&=[(g,1)][(e,-1)][(g^{-1},0)][(g,0)][(g^{-1},0)]\\
&=[(g,1)][(g^{-1}_1,-1)][(g,0)][(g^{-1},0)]\\
&=[(g,1)][(g^{-1}_1,-1)]=[(g,1)][(g,1)]^*,
\end{split}
\end{equation*}

and using $\mathcal{R}_3$ we see that it satisfies condition (iii).
\end{proof}

Consequently we have a $*$-homomorphism
\begin{equation}\label{eqpga2}
\begin{split}
\Psi: \mathds{U}[\varphi]&\rightarrow C^*_p(\overline{S},\mathcal{R})\\
                      u_g&\mapsto [(g,0)]\\
                      s^n&\mapsto [(e,n)].
\end{split}
\end{equation}

\begin{theorem}\label{teouppga} The C$^*$-algebra $\mathds{U}[\varphi]$ is isomorphic to $C^*_p(\overline{S},\mathcal{R})$.
\end{theorem}
\begin{proof} We just have to show that the $*$-homomorphisms (\ref{eqpga1}) and (\ref{eqpga2}) are inverses of each other on the generators of the respective C$^*$-algebras.

$\,\;\;\;\;\;\;\bullet\;\;\Phi\circ\Psi(u_g)=\Phi([(g,0)])=u_g$;

$\,\;\;\;\;\;\;\bullet\;\;\Phi\circ\Psi(s^n)=\Phi([(e,n)])=s^n$;
\vspace{0.3cm}\begin{equation*}
\begin{split}
\bullet\;\;\Psi\circ\Phi([(g_i,n)])&=\Psi({s^*}^iu_gs^{n+i})=[(e,-i)][(g,0)][(e,n+i)]\\
&=[(e,-i)][(g,0)][(e,n+i)][(e,-n-i)][(e,n+i)]\\
&=[(e,-i)][(g,n+i)][(e,-n-i)][(e,n+i)]\\
&=[(e,-i)][(e,i)][(e,-i)][(g,n+i)]\\
&=[(e,-i)][(e,i)][(\overline{\varphi}^{-i}(g),n)]\\
&=[(g_i,n)].
\end{split}
\end{equation*}
\end{proof}

In order to define a partial crossed product isomorphic to $C_p^*(\overline{S},\mathcal{R})$ which by the theorem above is isomorphic to $\mathds{U}[\varphi]$, consider $X_{\overline{S}}$ the subset of $\mathcal{P}(\overline{S})$ of the subsets $\xi$ of $\overline{S}$ which contain $(e,0)\in \overline{S}$. Also $1_s\in C(X_{\overline{S}})$ is given by
$$
1_s(\xi)=\left\{
           \begin{array}{ll}
             1, & s\in\xi; \\
             0, & \hbox{otherwise.}
           \end{array}
         \right.
$$

and the partial group algebra relations $\mathcal{R}$ are
\begin{itemize}
  \item[($\mathcal{R}_1$)] $e_{(g,\,0)}-1=0,\;\forall\; g\in G$;
  \item[($\mathcal{R}_2$)] $e_{(e,-n)}-1=0\;\forall\; n\in\mathds{N}$;
  \item[($\mathcal{R}_3$)] $\D\sum_{g\in\frac{G}{\varphi^n(G)}}e_{(g,n)}-1=0,\;\forall\; n\in\mathds{N}$.
\end{itemize}

This implies that $\widehat{\mathcal{R}}$ is the subset of $C(X_{\overline{S}})$ consisting of the functions
\begin{itemize}
  \item[($\widehat{\mathcal{R}}_1$)] $1_{(g,\,0)}-1_{(e,\,0)},\;\forall\; g\in G$;
  \item[($\widehat{\mathcal{R}}_2$)] $1_{(e,-n)}-1_{(e,\,0)}\;\forall\; n\in\mathds{N}$;
  \item[($\widehat{\mathcal{R}}_3$)] $\D\sum_{g\in\frac{G}{\varphi^n(G)}}1_{(g,n)}-1_{(e,\,0)},\;\forall\; n\in\mathds{N}$.
\end{itemize}

The spectrum of the relations $\mathcal{R}$ is defined to be
$$
\Omega_\mathcal{R}=\{\xi\in X_{\overline{S}}:\; f(g^{-1}\xi)=0,\;\forall\; f\in\widehat{\mathcal{R}},\;\forall\; g\in\xi\}.
$$

Consider
$$
\Omega_s=\{\xi\in\Omega_\mathcal{R}:\; s\in\xi\}
$$

and define the partial action $\varpi$ of $\overline{S}$ on $\Omega_\mathcal{R}$ by
\begin{equation}\label{isouapcp}
\begin{split}
\varpi_s: \Omega_{s^{-1}}&\rightarrow\Omega_s\\
                      \xi&\mapsto s\xi.
\end{split}
\end{equation}

Then it is well known by Theorem \ref{teouppga} and (\ref{parcpc14}) respectively that
\begin{equation}\label{isouapcp15}
\mathds{U}[\varphi]\cong C^*_p(\overline{S},\mathcal{R})\cong C(\Omega_\mathcal{R})\rtimes_{\alpha}\overline{S},
\end{equation}

where
\begin{equation}\label{isouapcp16}
\begin{split}
\alpha_s: C(\Omega_{s^{-1}})&\rightarrow C(\Omega_s)\\
                           f&\mapsto f\circ\varpi_{s^{-1}}.
\end{split}
\end{equation}

The partial crossed product description of $\mathds{U}[\varphi]$ presented above together with the requirement that $G$ is amenable (which implies that $\overline{S}$ is as well) makes it possible to define a certain conditional expectation as done in \cite{ExelFell} Proposition 2.9 (as in the classical group crossed product construction the amenability of the group implies the isomorphism of both reduced and full constructions by \cite{Mc}, and a faithful conditional expectation exists for the reduced one). We will show that this conditional expectation is the same - modulo the isomorphism already established - as $\epsilon$ as given by Proposition \ref{prop1}. The conditional expectation of $C(\Omega_\mathcal{R})\rtimes_{\alpha}\overline{S}$ is given by
\begin{equation*}
\begin{split}
\overline{E}: C(\Omega_\mathcal{R})\rtimes_{\alpha}\overline{S}&\rightarrow C(\Omega_\mathcal{R})\\
                                           f\delta_s           &\mapsto \left\{
                                    \begin{array}{ll}
                                      f, & \hbox{if }s=(e,0); \\
                                      0, & \hbox{otherwise.}
                                    \end{array}
                                  \right.
\end{split}
\end{equation*}
Identifying $C^*_p(\overline{S},\mathcal{R})$ with $C(\Omega_\mathcal{R})\rtimes_{\alpha}\overline{S}$,
$\overline{E}$ becomes
\begin{equation*}
\begin{split}
E: C^*_p(\overline{S},\mathcal{R})&\rightarrow C^*(e_{(g_i,n)})\footnotemark\\
                          \D\prod_{(g_i,n)\in\overline{S}}^{\scriptscriptstyle{\hbox{finite}}}[(g_i,n)]&\mapsto \left\{
                                    \begin{array}{ll}
                                      \D\prod_{(g_i,n)\in\overline{S}}^{\scriptscriptstyle{\hbox{finite}}}[(g_i,n)], & \hbox{if }\D\prod_{(g_i,n)\in\overline{S}}^{\scriptscriptstyle{\hbox{finite}}}(g_i,n)=(e,0); \\
                                      0, & \hbox{otherwise.}
                                    \end{array}
                                  \right.
\end{split}
\end{equation*}\footnotetext{$e_{(g_i,n)}:=[(g_i,n)][(g_i,n)^{-1}]$ with $(g_i,n)\in\overline{S}=\mathds{G}\rtimes_{\overline{\varphi}}\mathds{Z}$}

Using the isomorphism $\Psi$ (from (\ref{eqpga2})) and $\epsilon$ (from Proposition \ref{prop1}), we shall prove the following.
\begin{proposition} $E\circ\Psi=\Psi\circ\epsilon$.
\end{proposition}
\begin{proof} Let us prove the equality on the dense $*$-subalgebra of $\mathds{U}[\varphi]$ given by
$$
\hbox{span}(Q)=\hbox{span}(\{{s^*}^nu_{h^{-1}}fu_{h'}s^m:\;f\in P,\,h, h'\in G,\,n,m\in\mathds{N}\}).
$$

Consider $f=u_gs^k{s^*}^ku_{g^{-1}}\in P$, $h, h'\in G$, and $n,m\in\mathds{N}$.
\begin{equation*}
\begin{split}
&E\circ\Psi({s^*}^nu_{h^{-1}}fu_{h'}s^m)=E\circ\Psi({s^*}^nu_{h^{-1}}u_gs^k{s^*}^ku_{g^{-1}}u_{h'}s^m)\\
&= E([(e,-n)][(h^{-1},0)][(g,0)][(e,k)][(e,-k)][(g^{-1},0)][(h',0)][(e,m)])\\
&=\delta_{n,m}\delta_{h,h'}[(e,-n)][(h^{-1},0)][(g,0)][(e,k)][(e,-k)][(g^{-1},0)][(h,0)][(e,n)]\\
&=\delta_{n,m}\delta_{h,h'}[(e,-n)][(h^{-1},0)]\Psi(f)[(h,0)][(e,n)],
\end{split}
\end{equation*}

while
\begin{equation*}
\begin{split}
\Psi\circ\epsilon({s^*}^nu_{h^{-1}}fu_{h'}s^m)&=\Psi(\delta_{n,m}\delta_{h,h'}{s^n}^*u_{h^{-1}}fu_hs^n)\\
&=\delta_{n,m}\delta_{h,h'}[(e,-n)][(h^{-1},0)]\Psi(f)[(h,0)][(e,n)].
\end{split}
\end{equation*}

This shows that both conditional expectations $E$ and $\epsilon$ are the same, up to the isomorphism $\Psi$.
\end{proof}

\subsection{Simplicity of $\mathds{U}[\varphi]$}

To prove that $\mathds{U}[\varphi]$ is simple using partial crossed product theory, we suppose that $G$ is commutative. Therefore our group is amenable and the endomorphism $\varphi$ is totally normal i.e, the images of $\varphi$ are normal subgroups of $G$. This implies that the set $\overline{G}$, defined in the beginning of Section \ref{sectioncpdescr}, is a group.

We need some definitions (from \cite{ExLaQu}) concerning partial actions, as they play a role in the proof that $\mathds{U}[\varphi]$ is simple. Consider $(\theta,H,X)$ a partial dynamical system where $X$ is a locally compact space with $X_h$ being the open sets (Definition \ref{pcpdefi1}).
\begin{definition}We say that a partial action $\theta$ is topologically free if for every $h\in H\backslash\{e\}$ the set $F_h:=\{x\in X_{h^{-1}}:\;\theta_h(x)=x\}$ has empty interior.
\end{definition}

In order to define the minimality of $\theta$, we adjust the classical definition of invariance: a subset $V$ of $X$ is invariant under the partial action $(\theta,H,X)$ if $\theta_h(V\cap X_{h^{-1}})\subseteq V$ $\forall\; h\in H$.
\begin{definition}The partial action $\theta$ is minimal if there are no invariant open subsets of $X$ other than $\emptyset$ and $X$.
\end{definition}

Suited to our setting, there is a result due to Exel, Laca and Quigg (Corollary 2.9 of \cite{ExLaQu}) which says that the partial action $\varpi$ defined in (\ref{isouapcp}) is topologically free and minimal if and only if $C(\Omega_\mathcal{R})\rtimes_{\alpha}(\overline{S})$, as defined in (\ref{isouapcp15}) and (\ref{isouapcp16}), is simple (in fact their result applies to the reduced crossed product, but as we are assuming $G$ is commutative and thus amenable, we know that $\overline{S}$ is amenable and this implies that both the full and reduced partial crossed products are isomorphic by \cite{Mc} Proposition 4.2), so it is clear that we have to understand the topology of $\Omega_\mathcal{R}$, which unfortunately is not an easy task.

To avoid difficulties we present a new set which is homeomorphic to $\Omega_\mathcal{R}$, and for which we can easily understand the topology. Consider $\frac{G}{\varphi^k(G)}=\{e\}$ for negative integers $k$ and for $m\leq n$ both integers the canonical projection
$$
p_{m,n}: \dfrac{G}{\varphi^n(G)}\rightarrow\dfrac{G}{\varphi^m(G)}.
$$

Using these, define
\begin{equation*}
\begin{split}
\widetilde{G}:&=\lim_{\leftarrow \atop n}\left\{\dfrac{G}{\varphi^n(G)}:\;p_{m,n}\right\}\\
             &=\left\{(g_n\overline{\varphi}^n(G))_{n\in\mathds{Z}}\in\prod_{n\in\mathds{Z}}\dfrac{G}{\varphi^n(G)}:\;p_{m,n}(g_n)=g_m,\hbox{ if }m\leq n\right\},
\end{split}
\end{equation*}

where $\overline{\varphi}$ is the extension of $\varphi$ defined after Proposition \ref{propOre}. Note that when $n\leq0$, $\frac{G}{\varphi^n(G)}=\{e\}$ and therefore for any element in $\widetilde{G}$, the entries indexed by negative integers are $e$. Moreover, when $n>0$, $\overline{\varphi}^n=\varphi^n$. Particularly it makes not necessary to carry the bar over $\varphi$ when denoting the elements of $\widetilde{G}$, and we will also use the notation $(g_m)_{n\in\mathds{Z}}\in\widetilde{G}$. One can see $G$ inside $\widetilde{G}$ through the map $g\mapsto (g\varphi^n(G))_n$, which is injective if $\varphi$ is pure.

Another fact is that the set defined above is isomorphic as a topological group to our previous defined $\overline{G}$ (beginning of Section \ref{sectioncpdescr}), because that set is exactly this one except for the negative entries of the vectors in $\widetilde{G}$, which are always $e$. Therefore $\widetilde{G}$ is compact.

Consider
\begin{equation*}
\begin{split}
    \rho:\widetilde{G}&\rightarrow \hbox{P}(\overline{S})\\
(g_n\overline{\varphi}^n(G))_{n\in\mathds{Z}}&\mapsto\{(g_n\overline{\varphi}^n(h),n):\;n\in\mathds{Z},\;h\in G\}.
\end{split}
\end{equation*}

\begin{lemma} The set $\rho(\widetilde{G})$ is contained in $\Omega_\mathcal{R}$.
\end{lemma}
\begin{proof} Take $(g_m)_m\in\widetilde{G}$ and it is clear from the definition of $\widetilde{G}$ that $$g_m=g_{m-n}\overline{\varphi}^{m-n}(\overline{k}_1)$$ and $$g_{m+n}=g_m\overline{\varphi}^m(\overline{k}_2)$$ for $n\in\mathds{N}$ and $\overline{k}_1,\overline{k}_2\in G$.

Denote $\xi:=\rho((g_m)_m)$. We have to show that $f(g^{-1}\xi)=0$ for all $g\in\xi$ and all $f\in\widehat{\mathcal{R}}=\widehat{\mathcal{R}}_1\cup\widehat{\mathcal{R}}_2\cup\widehat{\mathcal{R}}_3$. Therefore fix $g=(g_m\overline{\varphi}^m(k),m)\in\xi$ for $m\in\mathds{Z}$ and $k\in G$.

$\bullet\; f=1_{(h,0)}-1\in\widehat{\mathcal{R}}_1$: Then $f(g^{-1}\xi)=0\Leftrightarrow g(h,0)\in\xi$, which is true because
$g(h,0)=(g_m\overline{\varphi}^m(kh),m)\in\xi$.

$\bullet\; f=1_{(e,-n)}-1\in\widehat{\mathcal{R}}_2$: Similarly $f(g^{-1}\xi)=0\Leftrightarrow g(e,-n)\in\xi$ and the latter holds as $g(e,-n)=(g_m\overline{\varphi}^m(k),m-n)=(g_{m-n}\overline{\varphi}^{m-n}(\overline{k}_1\overline{\varphi}^n(k)),m-n)\in\xi$.

$\bullet\; f=\D\sum_{h\in\frac{G}{\varphi^n(G)}}1_{(h,n)}-1\in\widehat{\mathcal{R}}_3$: Here $f(g^{-1}\xi)=0\Leftrightarrow$ there exists only one class $h\varphi^n(G)$ such that $g(h,n)\in\xi$. But
$$
g(h,n)=(g_m\overline{\varphi}^m(kh),m+n)=(g_{m+n}\overline{\varphi}^m(\overline{k}_2^{-1}kh),m+n)
$$

belongs to $\xi$ if and only if $\overline{k}_2^{-1}kh\in\overline{\varphi}^n(G)=\varphi^n(G)$ (as $n\in\mathds{N}$), which is the same as requiring $h\in k^{-1}\overline{k}_2\varphi^n(G)$, and this can be true only for one class in $\frac{G}{\varphi^n{G}}$.
\end{proof}

\begin{proposition}\label{homeorho} $\rho:\widetilde{G}\rightarrow \Omega_\mathcal{R}$ is a homeomorphism.
\end{proposition}
\begin{proof} If $\rho((g_m)_m)=\rho((h_m)_m)$ then $h_m=g_m\varphi^m(k_m)$ for all $m\in\mathds{N}$, with $k_m\in G$. Then $g_m=h_m$ in $\frac{G}{\varphi^m(G)}$ for all $m\in\mathds{N}$ and $(g_m)_m=(h_m)_m$ (note that for $m<0$, $g_m=h_m=e$).

Now let us prove that $\rho$ is surjective. Take $\xi\in\Omega_\mathcal{R}$ and remember that $(e,0)\in\xi$ which, using $f_1^h:=1_{(h,0)}-1\in\widehat{\mathcal{R}}_1$, implies that $(h,0)\in\xi$ $\forall\; h\in G$. Also for each $j\in\mathds{N}$, set $f_3^j:=\D\sum_{h\in\frac{G}{\varphi^j(G)}}1_{(h,j)}-1\in\widehat{\mathcal{R}}_3$.

As $f_3^j((e,0)\xi)=0$, for each $j$ there exists only one class $u_j\varphi^j(G)\in\frac{G}{\varphi^j(G)}$ such that $(u_j,j)\in\xi$. Using functions of the type $f_2^n:=1_{(0,-n)}-1\in\widehat{R}_2$, for $n\in\mathds{N}$, one sees that $(u_j\varphi^j(G))_{j\in\mathds{Z}}\in\widetilde{G}$. Now we prove that $\rho((u_j\varphi^j(G))_j)=\xi$.

By construction $(u_j,j)\in\xi$, which implies (using $f_1^h\in\widehat{\mathcal{R}}_1$ defined above) that $(u_j,j)(h,0)=(u_j\varphi^j(h),j)\in\xi$ for all $h\in G$. Doing the same for every $j$ it follows that $\rho((u_j\varphi^j(G))_j)\subseteq\xi$.

Suppose that $h=(k,i)\in\xi\backslash\rho((u_j\varphi^j(G))_j)$ and note that
$$
(k,i)\notin\rho((u_j\varphi^j(G))_j)\Leftrightarrow(k,i)\notin(u_i\varphi^i(G),i)\Leftrightarrow u^{-1}_ik\notin\varphi^i(G).
$$

Now consider the elements $g=(u_i,0)$ and $h'=(u_i,i)$ of $\rho((u_j\varphi^j(G))_j)\subseteq\xi$. Since $u^{-1}_ik\notin\varphi^i(G)$, we have that $g^{-1}h=(u^{-1}_ik,i)$ and $g^{-1}h'=(e,i)$ are different, which implies that $f_3^i(g^{-1}\xi)\neq 0$, and this contradicts the fact that $\xi\in\Omega_\mathcal{R}$.

Last, let us prove that $\rho$ preserves the topology. As the sets are compact and Hausdorff, it is enough to prove that $\rho^{-1}$ is continuous, which we will prove by showing that $\pi_m\circ\rho^{-1}$ is continuous for all $m\in\mathds{Z}$ where $\pi_m:\widetilde{G}\rightarrow\frac{G}{\varphi^m(G)}$ is the canonical projection.

As $\frac{G}{\varphi^m(G)}$ is discrete we just have to show that $\rho\circ\pi_m^{-1}(\{u_m\varphi^m(G)\})$ is open in $\Omega_R$ for all $u_m\varphi^m(G)\in\frac{G}{\varphi^m(G)}$. But note that (by the proof of surjectivity above)
$$
\rho\circ\pi_m^{-1}(\{u_m\varphi^m(G)\})=\{\xi\in\Omega_R:\;(u_m,m)\in\xi\},
$$

which is open in $\Omega_R$ (induced by the product topology in $\{0,1\}^{\overline{S}}$).

Then $\rho:\widetilde{G}\rightarrow \hbox{P}(\overline{S})$ is a homeomorphism.
\end{proof}

Using the proposition above, we identify $\Omega_\mathcal{R}$ with $\widetilde{G}$, and thus view $\varpi$ as a partial action of $G$ on $\widetilde{G}$. Remember that
$$
\Omega_s=\{\xi\in\Omega_\mathcal{R}:\; s\in\xi\}.
$$

Set
$$
\widetilde{G}_s:=\rho^{-1}(\Omega_s)
$$

and define
$$
\varpi_s:\widetilde{G}_{s^{-1}}\rightarrow\widetilde{G}_s.
$$

Using $\rho$ we can conclude that for $(g_i,n)\in\overline{S}=\mathds{G}\rtimes_{\overline{\varphi}}\mathds{Z}$ ($g_i\in G_i\hookrightarrow\mathds{G}$)
$$
\widetilde{G}_{(g_i,n)}=\{(h_m\varphi^m(G))_{m\in\mathds{Z}}\in\widetilde{G}:\;h_n\varphi^n(G)=g_i\varphi^n(G)\}
$$

(where $h_n$ is viewed inside $G=G_0\subseteq\mathds{G}$) and
$$
\varpi_{(g_i,n)}((h_m\varphi^m(G))_m)=(g_i\varphi^n(h_m)\varphi^{n+m}(G))_{n+m}=(g_i\varphi^n(h_{m-n})\varphi^m(G))_m.
$$

An easily proven and useful result follows.
\begin{lemma} \label{lema111}For $(g_i,n)\in\overline{S}$ the following holds:
\begin{enumerate}
  \item[(i)] $\widetilde{G}_{(g_i,n)}=\emptyset\Leftrightarrow g_i\notin G\varphi^n(G)$;
  \item[(ii)] $\widetilde{G}_{(g_i,n)}=\widetilde{G}\Leftrightarrow G\subseteq g_i\varphi^n(G)$.
\end{enumerate}
\end{lemma}
\begin{flushright}

  $\square$

  \end{flushright}

For $m\in\mathds{Z}$ and a subset $C_m\subseteq\frac{G}{\varphi^m(G)}$ (containing whole cosets) define the open set (it is open because it is the inverse image of a point via a projection)
$$
V_m^{C_m}=\{(u_n\varphi^n(G))_n\in\widetilde{G}:\;u_m\varphi^m(G)\in C_m\}.
$$

Clearly when $m\leq n$ then $V_m^{C_m}=V_n^{C_n}$ where
$$
C_n=\left\{u\varphi^n(G)\in\dfrac{G}{\varphi^n(G)}:\;u\varphi^m(G)\in C_m\right\}.
$$

From the definition of the product topology, we know that finite intersections of open sets $V_m^{C_m}$ form the base for the topology in $\widetilde{G}$. Since $V_{m_1}^{C_{m_1}}\cap V_{m_2}^{C_{m_2}}=V_{m}^{C_{n_1}}\cap V_{m}^{C_{n_2}}=V_m^{C_{n_1}\cap C_{n_2}}$ for $m\geq m_1,m_2$, $\left\{V_m^{C_m}\right\}$ is already a base for the topology.

Also note that if $C_m\neq\emptyset$ then for $k>0$, $C_{m+k}$ has at least 2 elements and therefore we can assume that if $V_m^{C_m}$ is not empty then $C_m$ has at least 2 elements (replacing $V_m^{C_m}$ by $V_n^{C_n}$ for $n>m$ if necessary).
\begin{proposition}\label{prop111}When $\varphi$ is a pure injective endomorphism of a commutative group $G$, the partial action $\varpi$ from $\widetilde{G}$ defined above is topologically free.
\end{proposition}
\begin{proof} Let us show that
$$
F_{(g_i,n)}=\{x\in\widetilde{G}_{(g_i,n)^{-1}}:\;\varpi_{(g_i,n)}(x)=x\}
$$

has empty interior, for $(g_i,n)\neq (e,0)$.

$\bullet$ Case 1: $n=0$. If $g_i\notin G$ then Lemma \ref{lema111} (i) assures that $F_{(g_i,0)}=\emptyset$.

Therefore suppose that $g_i\in G$. If $F_{(g_i,0)}\neq\emptyset$ the equation $\varpi_{(g_i,0)}(x)=x$ implies $g_i\in\varphi^m(G)$ for all $m\in\mathds{Z}$ (using the commutativity of $G$). As $\varphi$ is pure we conclude that $g_i=e$, and then $F_{(g_i,0)}=\emptyset$ for $g_i\neq e$.

$\bullet$ Case 2: Let $(g_i,n)$ with $n\neq 0$. Using again Lemma \ref{lema111} (i) we can assume that $g_i\in G\varphi^n(G)$. Take $V$ a non-empty open set of $\widetilde{G}_{(g_i,n)^{-1}}$ and, if needed, shrink $V$ so that $V=V_m^{C_m}$ (and we can assume that $m=ln>0$ for some big $l>0$). Note that we can assume that $C_m$ has at least 2 distinct elements, say $u_1\varphi^m(G)\neq u_2\varphi^m(G)$, which implies that $u_2^{-1}u_1\notin\varphi^m(G)$.

Suppose for a contradiction that $\varpi_{(g_i,n)}(x)=x$, $\forall\; x\in V$. Then, since $(u_j\varphi^k(G))_k\in V$ for $j=1,2$, we have
\begin{equation*}
\begin{split}
\varpi_{(g_i,n)}((u_j\varphi^k(G))_k)=(u_j\varphi^k(G))_k&\Rightarrow (g_i\varphi^n(u_j)\varphi^k(G))_k=(u_j\varphi^k(G))_k\\
&\Rightarrow u_j^{-1}g_i\varphi^n(u_j)\in\varphi^k(G)\hbox{ for }j=1,2\\
&\Rightarrow\varphi^n(u_2^{-1})u_2u_1^{-1}\varphi^n(u_1)\in\varphi^k(G), \forall\; k\in\mathds{Z}
\end{split}
\end{equation*}

(again we used the commutativity of $G$ to cancel the $g_i$'s). But as $\varphi$ is pure,
\begin{equation*}
\begin{split}
\varphi^n(u_2^{-1}u_1)=u_2^{-1}u_1&\Rightarrow\varphi^{ln}(u_2^{-1}u_1)=u_2^{-1}u_1\Rightarrow u_2^{-1}u_1\in\varphi^m(G)
\end{split}
\end{equation*}

which contradicts our hypothesis. So no open set can be contained in $F_{(g_i,n)}$, which implies that it has empty interior.
\end{proof}

\begin{proposition}\label{prop112}The partial action $\varpi$ is minimal.
\end{proposition}
\begin{proof} We will show that all $x\in\widetilde{G}$ has dense orbit by showing the following: if $V$ is a non-empty open set then there exists $(g,n)\in\overline{S}$ such that $x\in\widetilde{G}_{(g,n)^{-1}}$ and $\varpi_{(g,n)}(x)\in V$.

Take $x=(u_m\varphi^m(G))_{m\in\mathds{Z}}\in\widetilde{G}$ and $V=V_k^{C_k}\neq\emptyset$. Consider $u\varphi^k(G)\in C_k$ and define $(uu_k^{-1},0)$. By Lemma \ref{lema111} (ii), since $uu_k^{-1}G=G$, it follows that $\widetilde{G}_{(uu_k^{-1},0)^{-1}}=\widetilde{G}$ and therefore $x\in\widetilde{G}_{(uu_k^{-1},0)^{-1}}$.

To finish, note that
$$
\varpi_{(uu_k^{-1},0)}(x)=\varpi_{(uu_k^{-1},0)}((u_m\varphi^m(G))_m)=(uu_k^{-1}u_m\varphi^m(G))_m\in V.
$$
\end{proof}

We can now conclude (and this result agrees with the previous obtained Theorem \ref{teo1}):
\begin{theorem}If $\varphi$ is a pure injective endomorphism with finite cokernel of some commutative discrete countable group $G$ then the C$^*$-algebra $\mathds{U}[\varphi]$ is simple.
\end{theorem}
\begin{flushright}

  $\square$

  \end{flushright}

\begin{corollary}In the conditions of theorem above, we have$$C_r^*[\varphi]\cong\mathds{U}[\varphi].$$
\end{corollary}
\begin{flushright}

  $\square$

  \end{flushright}

\vspace{2cm}
\footnotesize DEPARTAMENTO DE MATEM\'{A}TICA - UFSC BLUMENAU - BRAZIL (f.vieira@ufsc.br)

\begin{thebibliography}{99}

\bibitem{BoEx} \textsc{G. Boava, R. Exel}, Partial crossed product description of the {C}$^*$-algebras
  associated with integral domains, \textit{Proceedings of the American Mathematical Society} \textbf{141} (2013), 2439--2451.

\bibitem{ChoiEffros} \textsc{M.-D. Choi and E. G. Effros}, Nuclear {C}$^*$-{A}lgebras and the {A}pproximation {P}roperty, \textit{Amer. J. Math.} \textbf{100}(1978), {61-79}.

\bibitem{CliPre} \textsc{A. H. Clifford, G. B. Preston}, The {A}lgebraic {T}heory of {S}emigroups, \textit{vol I, Mathematical Surveys, No. 7, Amer. Math. Soc.} (1996), Providence, RI.

\bibitem{Cuntz2} \textsc{J. Cuntz}, Simple {C}$^*$-algebras generated by isometries,
\textit{Comm. Math. Phys.} \textbf{85}(1977), 173--188.

\bibitem{CuntzTopMarkovII} \textsc{J. Cuntz}, A {C}lass of {C}$^*$-algebras and {T}opological {M}arkov {C}hains
  {II}: {R}educible {C}hains and the {E}xt-functor for {C}*-algebras, \textit{Inventiones mathematicae} \textbf{63}(1981), 25--40.

\bibitem{CuEcLi1} \textsc{J. Cuntz and S. Echterhoff, X. Li}, On the {K}-theory of crossed products by automorphic semigroup
  actions, \textit{Q. J. Math} \textbf{64(3)}(2013), 747--784.

\bibitem{Culi1} \textsc{J. Cuntz and X. Li}, The {R}egular {C}$^*$-algebra of an {I}ntegral {D}omain, \textit{Quanta of maths, Clay Math. Proc., Amer. Math. Soc., Providence, RI} \textbf{11}(2010), 149--170.

\bibitem{CunVer} \textsc{J. Cuntz and A. Vershik}, C$^*$-algebras associated with endomorphisms and polymorphisms of
  compact abelian groups, \textit{Comm. Math. Phys.} \textbf{321(1)}(2013), 157--179.

\bibitem{Exel1} \textsc{R. Exel}, Partial actions of groups and actions of inverse semigroups, \textit{Proc. Amer. Math. Soc.}\textbf{126} (1998), 3481--3494.

\bibitem{ExelFell} \textsc{R. Exel}, Amenability for {F}ell bundles, \textit{J. reine angew. Math.}, \textbf{492}(1997), 31--73.

\bibitem{ExLaQu} \textsc{R. Exel, M. Laca and J. Quigg}, Partial dynamical systems and {C}$^*$-algebras generated by partial isometries, \textit{J. Operator Theory} \textbf{47}(2002), 169--186.

\bibitem{ExVi} \textsc{R. Exel and F. Vieira}, Actions of inverse semigroups arising from partial actions of groups, \textit{J. Math. Anal. Appl.} \textbf{363}(2010), 86--96.

\bibitem{Hirsh} \textsc{I. Hirshberg}, On {C}$^*$-algebras associated to certain endomorphisms of discrete groups, \textit{New York J. Math.} \textbf{58}(2002), 99--109.

\bibitem{Khoska} \textsc{M. Khoshkam and G. Skandalis}, Toeplitz algebras associated with endomorphisms and {P}imsner-{V}oiculescu exact sequences, \textit{Pacific Journal of Mathematics} \textbf{181(2)}(1997), 315--331.

\bibitem{Kirchcla} \textsc{E. Kirchberg}, The classification of purely infinite {C}$^*$-algebras using {K}asparov's theory, \textit{notes}.

\bibitem{Laca1} \textsc{M. Laca}, From endomorphisms to automorphisms and back: dilations and full
  corners, \textit{Journal of the London Mathematical Society} \textbf{61}(2000), 893--904.

\bibitem{Laca2} \textsc{M. Laca, I. Raeburn}, Semigroup crossed products and the {T}oeplitz algebras of nonabelian groups, \textit{J. Funct. Anal.} \textbf{139}(1996), 415-440.

\bibitem{Li1} \textsc{X. Li}, Ring {C}$^*$-algebras, \textit{Math. Ann.} \textbf{348(4)}(2010), 859--898.

\bibitem{Mc} \textsc{K. McClanahan}, K-theory for partial crossed products by discrete groups, \textit{J. Functional Analysis}, \textbf{130} (1995), 77-117.

\bibitem{vNeu} \textsc{J. von Neumann}, Zur allgemeinen {T}heorie des {M}asses, \textit{Fund. {M}ath.} \textbf{13}(1929), 73--116.

\bibitem{Pivo1} \textsc{M. Pimsner, D. Voiculescu}, Exact sequences of {K}-groups and {E}xt-groups of certain crossed
  product {C}$^*$-algebras, \textit{J. Operator Theory} \textbf{4}(1980), 549--574.

\bibitem{QuRa} \textsc{J. Quigg and I. Raeburn}, Characterizations of crossed products by partial actions,\textit{J. Operator Theory} \textbf{37}(1997), 311--340.

\bibitem{Ror}, \textsc{M. R$\phi$rdam}, Classification of {N}uclear {C}$^*$-{A}lgebras, \textit{in Classification of {N}uclear {C}$^*$-{A}lgebras. {E}ntropy in {O}perator {A}lgebras, {E}ncyclopaedia of {M}athematical {S}ciences, {V}ol. 126, {S}pringer-{V}erlag, {B}erlin {H}eidelberg {N}ew {Y}ork}, (2002), New York.

\bibitem{Schoc} \textsc{C. Schochet}, Topological {M}ethods for {C}$^*$-{A}lgebras {II}: {G}eometric
  {R}esolutions and the {K}\"{u}nneth {F}ormula, \textit{Pacific Journal of Mathematic} \textbf{98}(1982), 443--458.

\bibitem{tutu} \textsc{J. L. Tu}, La conjecture de {B}aum-{C}onnes pour les feuilletages moyennable, \textit{K-Theory}, \textbf{17}(1999), 215-264.

\bibitem{Viei} \textsc{F. Vieira}, C$^*$-algebras generated by endomorphisms of groups, \textit{preprint} (2015).

\end{thebibliography}
\end{document}